\documentclass[english]{amsart}
\usepackage{babel}
\usepackage[]{geometry}
\usepackage{caption}
\usepackage{subcaption}

\usepackage{enumitem}
\usepackage{amsmath,amssymb,amsfonts,latexsym,cancel}
\usepackage{rawfonts}
\usepackage{pictexwd}
\usepackage{tikz}

\usepackage{tikz-cd}

\usepackage{pgfplots}

\pgfplotsset{compat=newest}
\pgfplotsset{plot coordinates/math parser=false}
\newlength\figureheight
\newlength\figurewidth 

\usepackage{color}
\usepackage{epsfig}

\newcommand{\R}{\mathbb{R}}
\newcommand{\N}{\mathbb{N}}

\def\1{\raisebox{2pt}{\rm{$\chi$}}}

\newcommand{\Lip}{\operatorname{Lip}}

\theoremstyle{plain}
\newtheorem{definition}{Definition}[]
\newtheorem{def-thm}{Definition-Theorem}[]
\newtheorem{proposition}{Proposition}
\newtheorem*{proposition ast}{Proposition \ref{Thm X_I charac}*}
\newtheorem{theorem}{Theorem}
\newtheorem{corollary}{Corollary}
\newtheorem{lemma}{Lemma}

\newtheorem{remark}{Remark}

\theoremstyle{definition}

\theoremstyle{remark}

\usepackage{hyperref}

\begin{document}

\title[Differentiability for Hamilton-Jacobi equations]{Differentiability with respect to the initial condition for Hamilton-Jacobi equations}

\begin{abstract}
We prove that the viscosity solution to a Hamilton-Jacobi equation with a smooth convex Hamiltonian of the form $H(x,p)$  is differentiable with respect to the initial condition. Moreover, the directional G\^ateaux derivatives can be explicitly computed almost everywhere in $\R^N$ by means of the optimality system of the associated optimal control problem.
We also prove that, in the one-dimensional case in space and in the quadratic case in any space dimension, these directional G\^ateaux derivatives actually correspond to the unique duality solution to the linear transport equation with discontinuous coefficient, resulting from the linearization of the Hamilton-Jacobi equation.
The motivation behind these differentiability results arises from the following optimal inverse-design problem: given a time horizon $T>0$ and a target function $u_T$,  construct an initial condition such that the corresponding viscosity solution at time $T$  minimizes the $L^2$-distance to $u_T$.
Our differentiability results allow us to derive a necessary first-order optimality condition for this optimization problem, and the implementation of gradient-based methods to numerically approximate the optimal inverse design.
\end{abstract}

\author{Carlos Esteve-Yag\"ue}
\thanks{AMS 2020 MSC: 35F21, 58C20, 35R30,49K20, 35Q49 
\textbf{Funding:} 
This project has received funding from the European Research Council (ERC) under the European Union’s Horizon 2020 research and innovation programme (grant agreement NO: 694126-DyCon),  the Alexander von Humboldt-Professorship program, the European Unions Horizon 2020 research and innovation programme under the Marie Sklodowska-Curie grant agreement No.765579-ConFlex, the Transregio 154 Project ‘‘Mathematical Modelling, Simulation and Optimization Using the Example of Gas Networks’’, project C08, of the German DFG, the Grant MTM2017-92996-C2-1-R COSNET of MINECO (Spain) and the Elkartek grant KK-2020/00091 CONVADP of the Basque government. 
}
\address {Carlos Esteve-Yag\"ue \newline \indent
{Departamento de Matem\'aticas, \newline \indent
Universidad Aut\'onoma de Madrid},
\newline \indent
{28049 Madrid, Spain}
\newline \indent \hspace{2.5cm} \text{and} \newline \indent
{Chair of Computational Mathematics, Fundaci\'on Deusto}
\newline \indent
{Av. de las Universidades, 24}
\newline \indent
{48007 Bilbao, Basque Country, Spain}
}
\email{\texttt{carlos.esteve@deusto.es}}

\author{Enrique Zuazua}
\address{Enrique Zuazua \newline \indent
{Chair for Dynamics,Control and Numerics -  Alexander von Humboldt-Professorship
\newline \indent
Department of Data Science, \newline \indent
Friedrich-Alexander-Universit\"at Erlangen-N\"urnberg}
\newline \indent
{91058 Erlangen, Germany}
\newline \indent \hspace{2.5cm} \text{and} \newline \indent
{Chair of Computational Mathematics, Fundación Deusto}
\newline \indent
{Av. de las Universidades, 24}
\newline \indent
{48007 Bilbao, Basque Country, Spain}
\newline \indent \hspace{2.5cm} \text{and} \newline \indent
{Departamento de Matemáticas, \newline \indent
Universidad Autónoma de Madrid,}
\newline \indent
{28049 Madrid, Spain}
}
\email{\texttt{enrique.zuazua@fau.de}}

\date{\today}

\maketitle

\section{Introduction}

We consider the initial-value problem associated to a Hamilton-Jacobi equation of the form
\begin{equation}\label{HJ eq}
\left\{\begin{array}{ll}
\partial_t u + H (x,\nabla_x u) = 0 & \text{in} \ (0,T)\times \R^N \\
u(0, \cdot)= u_0 & \text{in} \ \R^N.
\end{array}\right.
\end{equation}
where $T>0$,  $u_0\in \Lip (\R^N)$ is the given initial condition, and $H$ is the given Hamiltonian
 $$H:\R^N \times \R^N \to \R,$$
 which will be assumed to satisfy the following hypotheses: 
\begin{equation}\label{general Hcond}
\tag{H1}
H\in C^2(\R^{2N}), \quad
c_0I_N \leq H_{pp}(x, p) \quad  \forall x,p\in \R^N,  \quad \text{for some constant $c_0>0$},
\end{equation}
\begin{equation}\label{H bounded}
\tag{H2}
H(x,p) \geq -C_0
\quad \forall x,p\in \R^N \quad \text{and} \quad H(x,0) \leq C_0
 \quad \forall x\in \R^N,
 \quad
 \text{for some constant $C_0\geq 0$},
\end{equation}
and that $H$ satisfies the two following Lipschitz estimates:
 there exists $C_{\Lip} >0$ such that
\begin{equation}\label{H lipschitz in x}
\tag{H3}
|H_x(x,p)| \leq C_{\Lip} \qquad \forall (x,p) \in \R^N\times \R^N,
\end{equation}
and for every $R>0$, there exists $C_R>0$ such that
\begin{equation}\label{H lipschitz in p}
\tag{H4}
|H_p (x,p)| \leq C_R, \qquad \forall (x,p)\in\R^N\times \R^N, \quad \text{s.t.} \quad |p| \leq R.
\end{equation}
In \eqref{general Hcond}, $I_N$ denotes the $N\times N$ identity matrix, and the  inequality is understood in the usual partial order of squared symmetric matrices, i.e. $A\geq B$ means that $A-B$ is definite positive.

The study of Hamilton-Jacobi equations such as \eqref{HJ eq} arises in the context of optimal control theory and problems in calculus of variations, in which the so-called value function satisfies an associated dynamic programming equation, also known as Bellman equation \cite{bellman1956dynamic}, which, in the deterministic continuous setting, happens to be equivalent to a first-order non-linear partial differential equation of the form \eqref{HJ eq}, see for instance \cite[Chapter 10]{evans1998partial} or \cite{fleming2006controlled}.
We also mention that Hamilton-Jacobi equations are intimately connected, via the dynamic programming principle,  to problems in reinforcement learning \cite{bertsekas2019reinforcement},  where many algorithms consist in approximating the value function associated to an optimal control problem by exploiting the fact that it satisfies the associated dynamic programming equation (the Bellman equation).

By the classical theory of viscosity solutions \cite{crandall1992user,crandall1983viscosity,lions1982generalized}, it is well-known that, for any initial condition $u_0\in \Lip (\R^N)$, there exists a unique solution $u\in \Lip  ([0,T]\times \R^N)$ which satisfies \eqref{HJ eq} in the viscosity sense. 
Moreover,  the convexity hypothesis \eqref{general Hcond} on the Hamiltonian induces a regularizing effect on the viscosity solution, which makes it be semiconcave with respect to $(t,x)$ (see \cite{barron1999regularity,bianchini2012sbv, cannarsa2014pointwise,  cannarsa1997regularity, lions2020new} for results concerning the regularity of viscosity solutions to Hamilton-Jacobi equations).
In this work,  we study the differentiability of the viscosity solution to \eqref{HJ eq} with respect to the initial condition $u_0$.
In particular, we shall address the following issues concerning the initial-value problem \eqref{HJ eq}:
 
\begin{enumerate}[label*=\arabic*.]
\item Our main goal is to establish the differentiability of the viscosity solution to \eqref{HJ eq} with respect to the initial condition $u_0$. More precisely,  for any $T>0$ and $u_0, w \in \Lip(\R^N)$, we are interested in the existence of the directional G\^ateaux derivative of the solution $u(t,x)$ at the initial condition $u_0$, in the direction $w$, that we shall denote by $\partial_w u$, and is defined as the limit 
$$
\partial_w u (t,\cdot) : = \lim_{\delta\to 0^+} \dfrac{u_\delta (t,\cdot) - u(t,\cdot)}{\delta} \qquad \text{as} \  \delta\to 0^+,\qquad \text{for} \ t\in [0,T],
$$
where, for each $\delta >0$, $u_\delta$ and $u$ are the viscosity solutions to \eqref{HJ eq} with initial conditions $u_0 + \delta w$ and $u_0$ respectively.
\item The motivation behind the previous point arises in the context of inverse-design problems associated to \eqref{HJ eq}, for which differentiability results may allow us to derive first-order optimality conditions and the implementation of gradient-based algorithms to numerically approximate an optimal inverse design.  
In particular, we address the following inverse-design problem: for a given time horizon $T>0$ and a target function $u_T\in \Lip(\R^N)$, construct an initial condition $u_0$ such that the corresponding viscosity solution to \eqref{HJ eq} at time $T$ minimizes the $L^2$-distance to the target function $u_T$. 
This problem can also be cast as the orthogonal projection of $u_T$ onto the reachable set $\mathcal{R}_T$,  that we define as
\begin{equation}
\label{def: reachable set}
\mathcal{R}_T := \left\{ u_T\in \Lip(\R^N)\ : \quad \exists u_0 \in \Lip(\R^N) \ \text{s.t.  the solution $u$ to \eqref{HJ eq} satisfies} \ u(T,\cdot) = u_T  \right\}
\end{equation}
i.e., the set of functions $u_T$ such that there exists at least an initial condition $u_0$ for which the viscosity solution to \eqref{HJ eq} coincides with $u_T$ at time $T$.
\end{enumerate}

The inverse-design problem described in the second point above can be seen as an optimal control problem subject to the dynamics given by the Hamilton-Jacobi equation \eqref{HJ eq}, where the control is the initial condition $u_0$.
Optimal control problems associated to first-order Hamilton-Jacobi equations such as \eqref{HJ eq} are much less studied in the literature compared to optimal control problems subject to other type of PDEs as for instance second-order parabolic and elliptic equations.
This is mostly due to the fact that the adjoint formulation to compute the gradient of the functional presents difficulties when the gradient of the solution develops discontinuities, and the solution ceases to exist in the classical sense.

For the case of scalar conservation laws in dimension 1, optimal control problems have been considered in \cite{bressan2007optimality, colombo2004optimization, ulbrich2002sensitivity, ulbrich2003adjoint}.  In their approach, the variation of the functional with respect to the control is computed by using the notion of \emph{shift differentiability}, introduced by Bressan and Guerra in \cite{bressan1997shift},  and extended to the multidimensional case in \cite{bianchini2000shift, bressan1999shift}.
The notion of shift derivative on the space $BV$ of integrable functions of bounded variation provides differentiability for the solution operator associated to scalar conservation laws by carefully measuring the sensitivity of the shock discontinuities.
Despite the close relation between scalar conservation laws and Hamilton-Jacobi equations \cite{colombo2019initial}, we do not make use of the notion of shift derivative since it requires a precise description of the singular set, which turns out to be rather complicated, specially in the multi-dimensional case, and might lead us to make further regularity assumptions.
Instead, we establish the differentiability of the solution operator associated to \eqref{HJ eq} by using the fact that the viscosity solution can be written as the value function of an associated problem in calculus of variations.  
In the one-dimensional case in space, and for quadratic Hamiltonians of the form $H(x,p) = |p|^2 + f(x)$ in any space dimension, we can also use the differentiability result to provide a dual (or adjoint) formulation of the gradient of the functional.
A  similar approach to ours, which is based on the adjoint equation (and does not make use of the notion of shift differentiability) is used in \cite{bouchut1999differentiability,colombo2011control} for scalar conservation laws in dimension 1.

\subsection{Contributions}

The main results in this work concern the differentiability of the viscosity solution to \eqref{HJ eq} with respect to the initial condition $u_0$. 
In view of the well-posedness of the initial value problem \eqref{HJ eq},  for any $t\in (0,T]$, we can define the nonlinear operator
\begin{equation}\label{forward viscosity operator intro}
\begin{array}{cccc}
S_t^+ : & \Lip(\R^N) & \longrightarrow & \Lip(\R^N) \\
 & u_0 &\longmapsto & u(t,\cdot),
\end{array}
\end{equation}
which associates, to any initial condition $u_0$, the viscosity solution to \eqref{HJ eq} at time $t$.
Our goal is therefore to establish the differentiability of the operator $S_t^+$.
Here we sum up our main contributions concerning this issue:

\begin{enumerate}
\item In Theorem \ref{thm: differentiability general weak convergence}, we prove that for any $t\in (0,T]$, the operator $S_t^+$ is G\^ateaux differentiable at any $u_0\in \Lip(\R^N)$,  in any direction $w\in \Lip(\R^N)$, with respect to the $L^1_{loc}$--convergence. Namely, we prove that for any $u_0,w\in \Lip(\R^N)$, there exists a function $\partial_w S_t^+ u_0\in L^\infty_{loc}(\R^N)$ such that
$$
\dfrac{S_t^+ (u_0 + \delta w) - S_t^+ u_0}{\delta} \longrightarrow \partial_w S_t^+ u_0, \qquad \text{as} \quad \delta\to 0^+, \quad \text{in} \quad L^1_{loc} (\R^N).
$$
Moreover, the function $\partial_w S_t^+ u_0$ can be explicitly computed almost everywhere in $\R^N$ by means of the optimality system of the optimal control problem associated to the Hamilton-Jacobi equation \eqref{HJ eq}.

\item Next, for the one-dimensional case in space and for qudratic Hamiltonians of the form
$$
H(p,x) = |p|^2 + f(x),
$$
in any space dimension, we prove in  Theorem \ref{thm:gateaux deriv transport} that the function $v:(0,T]\times\R^N \longrightarrow \R$, defined as
$$
v(t,x) := \partial_w S_t^+ u_0(x) \qquad \forall (t,x) \in (0,T]\times \R^N
$$ 
is the unique duality solution \cite{bouchut1998one,bouchut2005uniqueness} to the linear transport equation 
\begin{equation}\label{eq: transport eq intro Ham}
\left\{\begin{array}{ll}
\partial_t v + H_p(x, \nabla_x u )\cdot \nabla_x v = 0 & (t,x) \in (0,T)\times\R^N \\
v(0,\cdot ) = w & x\in \R^N,
\end{array}\right.
\end{equation}
where $u$ is the viscosity solution to \eqref{HJ eq}.
Note that, due to the low regularity of the viscosity solution $u$, the transport coefficient in \eqref{eq: transport eq intro Ham} might have discontinuities.
\end{enumerate}

\begin{remark}[Transport equations with  one-sided-Lipschitz coefficient]
In particular, the conclusions of Theorem \ref{thm:gateaux deriv transport} establish the existence and uniqueness of a duality solution for the linear transport equation \eqref{eq: transport eq intro Ham}, for any  initial condition $w\in \Lip(\R^N)$. 
For this purpose, we restrict ourselves to the one-dimensional case in space and to quadratic Hamiltonians of the form
\begin{equation}
\label{quadratic hamiltonian intro}
H(x,p) = |p|^2 + f(x),
\end{equation}
where $f\in C^2 (\R^N)$ is bounded and globally Lipschitz.
The restriction arises since, in order to provide uniqueness for the duality solution to \eqref{eq: transport eq intro Ham},  we need to ensure that the transport coefficient
\begin{equation}
\label{transport coeff intro}
a(t,x) = H_p(x, \nabla_x u ) \in \big[L^\infty ((0,T)\times \R)  \big]^N
\end{equation}
satisfies the one-sided Lipschitz condition (OSLC)
\begin{equation}\label{OSLC intro}
 \langle a(t,y)-a(t,x), \, y-x \rangle  \leq \alpha(t)  |y-x|^2  \qquad \text{for a.e.} \  (t,x,y)\in (0,T)\times\R^N\times\R^N,
\end{equation}
with $\alpha(t) = C/t$ for some constant $C = C(u_0, H)>0$ depending on $u_0$ and $H$. 
In the one-dimensional case in space and for quadratic Hamiltonians in any space dimension, this property follows from the uniform convexity of the Hamiltonian $H$ and the semiconcavity of the viscosity solution (see Proposition \ref{prop: semiconcavity property}).
For the case of general convex Hamiltonians in multiple dimensions, the semiconcavity of the solution does not in general imply the one-sided-Lipschitz condition on the transport coefficient $a(t,x)$, and therefore, the uniqueness of the duality solution to \eqref{eq: transport eq intro Ham} is not straightforward.

Observe that \eqref{OSLC intro} implies only an upper bound for $\operatorname{div}_x a$, and thus, $\operatorname{div}_x a$ may not be absolutely continuous with respect to the Lebesgue measure, preventing us from using the approach by DiPerna-Lions in \cite{diperna1989ordinary} or by Ambrosio in  \cite{ambrosio2004transport}.
The notion of duality solution, as introduced by Bouchut and James in \cite{bouchut1998one} for the one-dimensional case,   and by Bouchut-James-Mancini in \cite{bouchut2005uniqueness} for the $N$-dimensional case,  provides existence, uniqueness and stability for the initial-value problem associated to linear transport equations under a OSLC condition slightly stronger than \eqref{OSLC intro}. 
Their proof relies on the well-posedness of the backward dual problem (which is a conservative transport equation), in the sense of reversible solutions (see subsection \ref{subsec: duality solutions} for further details).
However, in \cite{bouchut1998one,bouchut2005uniqueness}, it is  assumed that the function $\alpha$ in \eqref{OSLC intro} satisfies $\alpha\in L^1(0,T)$,  which is not fulfilled in our case.
In Theorem \ref{thm:gateaux deriv transport}, we are able to overcome this difficulty and prove existence and uniqueness of a duality solution for linear transport equations, when the transport coefficient $a(t,x)$ has the form \eqref{transport coeff intro} for some Hamiltonian $H$ satisfying \eqref{general Hcond}, \eqref{H bounded}, \eqref{H lipschitz in x}, and when the initial condition is continuous. 
However, our proof only applies to the one-dimensional case and to the multi-dimensional case when $H$ has the form \eqref{quadratic hamiltonian intro}.
Our proof relies on the possibility of uniquely extending the reversible solutions to the backward dual problem  by a measure at time $t=0$ (see Proposition \ref{prop:reversible unique extension} in subsection \ref{subsec: proof linearization}). 
The fact that the backward solutions can only be extended at $t=0$ by a measure restricts the well-posedness result for the forward equation to the case of continuous initial conditions.
\end{remark}

Let us now turn our attention to the second goal of this work, which concerns the optimal inverse-design problem associated to \eqref{HJ eq}.
This problem can be formulated as an optimal control problem in which the dynamics are given by the Hamilton-Jacobi equation \eqref{HJ eq}, and the control is, precisely, the initial condition $u_0$, i.e.
\begin{equation}\label{L2 OCP}
\begin{array}{l}
\underset{u_0\in \Lip_0(\R^N)}{\text{minimize}} \ \mathcal{J}_T (u_0) : =  \| u(T,\cdot) - u_T(\cdot) \|_{L^2 (\R^N)}^2 \\
\qquad 
\begin{array}{ll}
 \text{s.t. } & \text{$u$ is the viscosity solution to \eqref{HJ eq},} \\
& \text{with initial condition $u_0$.}
\end{array}
\end{array}
\end{equation}
Here, $T>0$ and $u_T\in \Lip_0(\R^N)$ are the given time-horizon and  target function respectively. 
The notation $\Lip_0(\R^N)$ stands for the space of Lipschitz functions with compact support in $\R^N$.

In this case, the assumptions on the Hamiltonian are \eqref{general Hcond} and \eqref{H lipschitz in x} just as before, but this time, assumption \eqref{H bounded} is assumed to hold with $C_0=0$.
This choice guarantees that the operator $S_T^+$ satisfies
$$
S_T^+ u_0(\cdot) \in \Lip_0 (\R^N) \qquad
\forall u_0 \in \Lip_0(\R^N),
$$  
which, along with the fact that the target $u_T$ is compactly supported, ensures that the set of admissible controls (initial conditions $u_0$ such that $\mathcal{J}_T (u_0) < \infty$) is nonempty (see Remark \ref{rmk: C=0}).
Our conclusions concerning the optimal control problem \eqref{L2 OCP} are as follows:

\begin{enumerate}
\item In Theorem \ref{thm:optimality condition}, we use the differentiability result from Theorem \ref{thm: differentiability general weak convergence} to derive a first-order optimality condition for the optimization problem \eqref{L2 OCP}, which can be expressed by means of the gradient of the functional  $\mathcal{J}_T$, i.e. the linear functional
$$
D \mathcal{J}_T (u_0) : w\in \Lip(\R^N) \longmapsto \partial_w \mathcal{J}_T (u_0),
$$
which defines a Radon measure in $\R^N$.
In Theorem \ref{thm:optimality condition}, the gradient of $\mathcal{J}_T$ is given explicitly by using the optimality system of the optimal control problem associated to the Hamilton-Jacobi equation \eqref{HJ eq}.
 
\item Then,  for the one-dimensional case in space, and for quadratic Hamiltonians of the form \eqref{quadratic hamiltonian intro} in any space dimension, we prove in Theorem \ref{thm: opt cond dual equation}  that, for any $u_0,w\in \Lip(\R^N)$, the directional derivative of the functional $\mathcal{J}_T$ at $u_0$ in the direction $w$ can be given by duality as
\begin{equation}
\label{directional derivative J_T intro}
\partial_w \mathcal{J}_T (u_0) = 2\displaystyle\int_{\R^N} w(x) d\pi^0(x),
\end{equation}
where $\pi^0\in \mathcal{M}(\R^N)$ is the unique Radon measure  which continuously\footnote{In Proposition \ref{prop:reversible unique extension}, we prove that the unique reversible solution $\pi\in C((0,T]; \, L^1_{loc} (\R^N)$ to \eqref{backward conservative intro} converges,  as $t\to 0^+$, to a Radon measure $\pi^0$ in the $\text{weak}^\ast$ topology in the space of measures.} extends  at $t=0$ the unique reversible solution to the backward conservative transport equation
\begin{equation}
\label{backward conservative intro}
\left\{\begin{array}{ll}
\partial_t \pi + \text{div}_x (a(t,x) \pi ) = 0 & \text{in} \ (0,T)\times \R^N, \\
\pi (T,x) = S_T^+u_0 (x) - u_T(x) & \text{in} \ \R^N,
\end{array}\right.
\end{equation}
where the transport coefficient $a$ is defined as in \eqref{transport coeff intro}.
This allows us to derive the same first-order optimality condition as in Theorem \ref{thm:optimality condition}, but this time the gradient of the functional $\mathcal{J}_T$ is given by the Radon measure
$D\mathcal{J}_T(u_0) = 2\pi^0$.

\item For completeness purposes, in Theorem \ref{thm:existence L2 projection}, we prove that the inverse design problem \eqref{L2 OCP} admits at least one solution.  We point out that uniqueness of a minimizer is not true in general due to the lack of backward uniqueness of the Hamilton-Jacobi equation \eqref{HJ eq}, see \cite{colombo2019initial,esteve2020inverse}.
The proof of the existence result relies on a compactness argument and utilizes the so-called backward viscosity operator \cite{barron1999regularity}, denoted by $S_T^-$, which is defined in an analogous way to $S_T^+$.
\item Finally, we look at the related optimization problem of the orthogonal projection of $u_T\in \Lip_0(\R^N)$ onto the reachable set $\mathcal{R}_T$, which can be formulated as
$$
\underset{\varphi\in \mathcal{R}_T}{\text{minimize}} \ \mathcal{H}_T (u_0) : =  \| \varphi - u_T \|_{L^2 (\R^N)}^2.
$$
For this problem,  under the assumption that the Hamiltonian $H$ is $x$-independent and quadratic, we are able to prove, in Theorem \ref{thm:existence and uniqueness}, existence and uniqueness of a minimizer. 
In this case, the proof uses Hilbert's projection Theorem, that we can apply since, for $x$-independent quadratic Hamiltonians,  the reachable set $\mathcal{R}_T$ can be fully characterized by means of a sharp semiconcavity estimate from \cite{esteve2020inverse,lions2020new}, that allows us to prove that the set $\mathcal{R}_T\cap L^2(\R^N)$ is closed and convex in $L^2(\R^N)$.
Note that, although the orthogonal projection of $u_T$ onto $\mathcal{R}_T$ is unique, the solution to the inverse design problem \eqref{L2 OCP} is not expected to be unique, not even for $x$-independent quadratic Hamiltonian. In Corollary \ref{cor:set of inverse designs} we give a full characterization of the solutions to \eqref{L2 OCP} when $H$ is quadratic and independent of $x$.
\end{enumerate}

\begin{remark}[Gradient-based methods]
\label{rmk:gradient-method}
A popular method to numerically approximate a solution to optimal control problems such as \eqref{L2 OCP} is the so-called \emph{gradient-descent algorithm}, which consists in  repeatedly updating the control parameter (in this case $u_0$) in the opposite direction to the gradient of the functional to be minimized (in this case $\mathcal{J}_T$).
Note however that, in view of Theorems \ref{thm:optimality condition} and \ref{thm: opt cond dual equation}, the gradient $D \mathcal{J}_T (u_0)$ is a Radon measure in $\R^N$, and then,  the process of updating the initial condition in the opposite direction to the gradient might not be possible, as it would exit the space $\Lip(\R^N)$.
We can nonetheless implement an approximate version of the gradient descent algorithm, in which at each step,  the initial condition is updated, not in the exact opposite direction to the gradient, but in a Lipschitz approximation of it, $\widehat{v}$, making sure that the directional derivative, as defined in \eqref{directional derivative J_T intro}, satisfies $\partial_{\widehat{v}} \mathcal{J}_T(u_0) > 0$.

The method is initialized with some initial condition $u_0^{(0)}\in \Lip_0(\R^N)$, and then, it is updated at each step by means of the following formula:
\begin{equation}
\label{gradient descent}
u_0^{(n+1)} := u_0^{(n)} - \gamma_n  \dfrac{\widehat{v}_n}{\| \widehat{v}_n\|_{\infty}},
\end{equation}
where $\widehat{v}_n$ is a Lipschitz approximation of the Radon measure $D \mathcal{J}_T (u_0^{(n)})$, and $\| \widehat{v}_n\|_\infty$ stands for the sup-norm of $\widehat{v}_n$.
The step-size parameter $\gamma_n\in (0,1)$ can be chosen in an adaptive manner depending on $n$ to ensure the convergence of the algorithm.
Besides, it is to be pointed out that, since the functional
$$
u_0 \longmapsto \mathcal{J}_T (u_0) = 
\| S_T^+ u_0(\cdot) - u_T(\cdot) \|_{L^2(\R^N)}^2
$$
is not necessarily convex, the gradient descent algorithm might converge to a local minimum instead of a global one.
\end{remark}

In Figure \ref{Fig:inverse-designs}, we illustrate the application of the approximate gradient descent algorithm described in Remark \ref{rmk:gradient-method}, to the functional $\mathcal{J}_T$, when the target $u_T$ is unreachable. As initialization, we have chosen the initial condition given by the backward viscosity operator applied to the target $u_T$, i.e. $u_0^{(0)} := S_T^- u_T$, which is represented in the plot at the left.
In the center, we plotted the initial condition $\widehat{u}_0$, obtained after several steps of the approximate gradient descent algorithm \eqref{gradient descent}.
The picks of the function $\widehat{u}_0$ pointing downward  arise since, at each step, we are approximating the gradient $D\mathcal{J}_T(u_0^{(n)})$, which is in general a Radon measure,  by a Lipschitz function. 
In the plot at the right, we see the function
$$\tilde{u}_0 : = S_T^- (S_T^+ \widehat{v}_0),$$
which, in view of the well-known property
$S_T^+ (S_T^-( S_T^+ u_0)) = S_T^+ u_0$, which holds for all $u_0\in \Lip(\R^N)$, see \cite{barron1999regularity,misztela2020initial},  we have that
$$
\mathcal{J}_T (\tilde{u}_0) = \mathcal{J}_T(\widehat{u}_0).
$$
Then, the function $\tilde{u}_0$ can be seen as a regularized version of the (approximately) optimal inverse design $\widehat{u}_0$.  This picture also shows how the minimizers for the optimal control problem \eqref{L2 OCP} are not expected to be unique. 
In Figure \ref{Fig:L2 projection},  we see the image by $S_T^+$ of the three initial conditions depicted in Figure \ref{Fig:inverse-designs}: in the plot at the left, we see the function $S_T^+ (S_T^- u_T)$,  whereas the plot at the right represents the functions $S_T^+ \widehat{u}_0$ and $S_T^+ \tilde{u}_0$, which are indeed equal.

\begin{center}
\begin{figure}
\includegraphics[scale=.18]{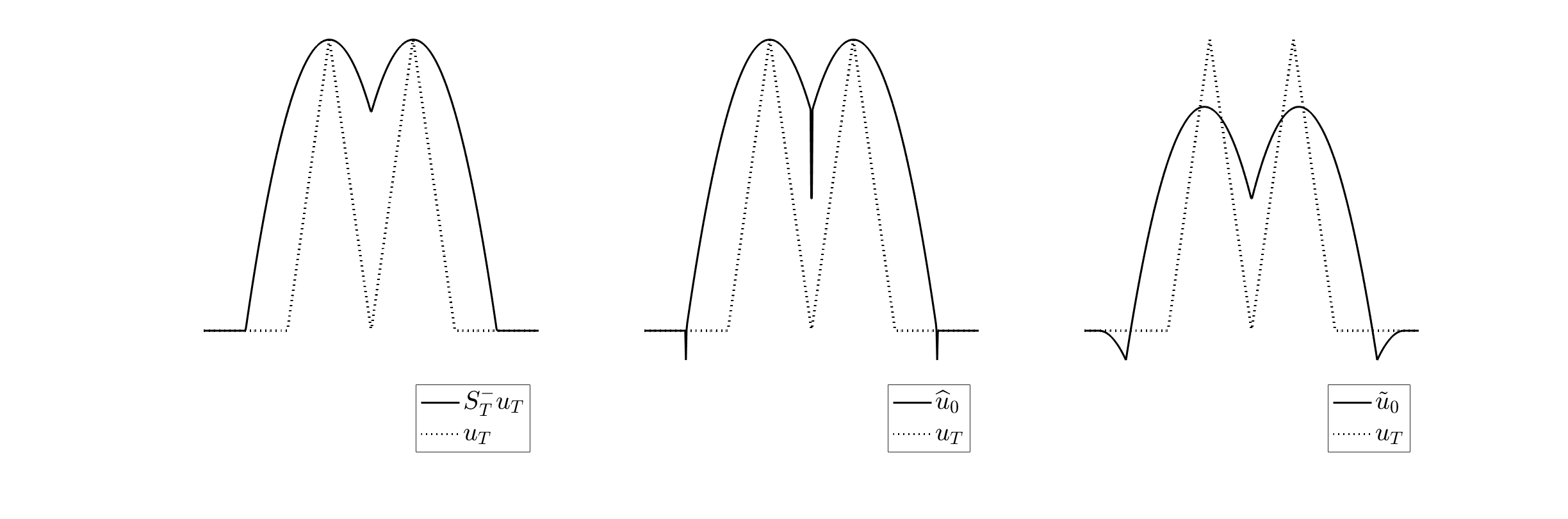}
\caption{Different inverse designs for an unreachable target. 
At the left we see the inverse design obtained as the backward viscosity solution to \eqref{HJ eq} with terminal condition $u_T$. 
In the center we have the initial condition after several iterations of the approximate gradient descent algorithm described in Remark \ref{rmk:gradient-method}.
At the right, we see the function $S_T^- (S_T^+ \widehat{u})$, which is a regularized version of the (approximately) optimal inverse design depicted in the center.}
\label{Fig:inverse-designs}
\end{figure}
\end{center}

\subsection{Previous related results on inverse design for Hamilton-Jacobi equations}

One of the main motivations of the present work arises in the context of inverse-time design problems associated to \eqref{HJ eq}, in which, for a given (possibly noisy) observation $u_T$ of the solution at some time $T>0$, one aims at reconstructing the corresponding initial condition.
In the recent works \cite{colombo2019initial,esteve2020inverse}, it is shown that  this inverse problem is highly ill-posed. 
In one hand, the given observation $u_T$ might not be reachable, meaning that there exists no initial condition for which the viscosity solution coincides with $u_T$ at time $T$.
On the other hand, when an inverse design for $u_T$ exists, it might not be unique.

In \cite{esteve2020inverse}, we treated the case of  $x$-independent Hamiltonians $H(x,p) = H(p)$ satisfying the following properties
\begin{equation}\label{Hcond previous results}
H\in C^2(\R^N), \qquad H_{pp} (p) >0 \quad \forall p\in \R^N \quad \text{and} \quad \lim_{|p|\to \infty} \dfrac{H(p)}{|p|} = +\infty.
\end{equation}
We showed that, in this case,  existence and uniqueness of inverse designs are intimately related to the regularity of the target $u_T$.
More precisely, the existence of at least an inverse design depends on the semiconcavity properties of $u_T$, whereas the uniqueness is related to its differentiability.

In the case when there exists no initial condition $u_0\in \Lip(\R^N)$ such that $S_T^+ u_0 = u_T$,  the question that arises naturally is that of constructing an initial condition $\widehat{u}_0\in \Lip(\R^N)$ such that $S_T^+ \widehat{u}_0$ is ``as close as possible'' to $u_T$. 
Of course, one can consider different criteria to measure the closeness of $S_T^+ \widehat{u}_0$ to the target $u_T$.
In this work we consider the $L^2$-distance, but other choices can be considered.  We refer to \cite{liardanalysis,liard2021initial} for similar results for the one-dimensional Burgers equation with convex flux, and to \cite{refId0} for a one-dimensional conservation law with a discontinuous space-dependent flux.

For the case of $x$-independent Hamiltonians satisfying \eqref{Hcond previous results}, we studied in \cite{esteve2020inverse} the reachable function $u_T^\ast$ obtained after a backward-forward resolution of \eqref{HJ eq} with terminal condition $u_T$.
This method gives rise to the smallest element in $\mathcal{R}_T$ which is bounded from below by $u_T$, the so-called semiconcave envelope of $u_T$. 
Hence, the inverse design $u_0 = S_T^- u_T$, is optimal if one wants to approximate $u_T$ with functions lying above $u_T$.
Moreover,  for the case of $x$-independent quadratic Hamiltonians of the form
$$
H(p) = \dfrac{\langle Ax, \, x\rangle}{2} \qquad \text{with $A$ being any positive definite matrix,}
$$
we proved in \cite{esteve2020inverse} that the function $u^\ast_T : = S_T^+ (S_T^- u_T)$, obtained after a backward-forward resolution of \eqref{HJ eq} in the time interval $[0,T]$, actually coincides with the unique viscosity solution to the degenerate elliptic obstacle problem
$$
\min \left\{ \varphi - u_T, \ -\lambda_N \left[ D^2 \varphi - \dfrac{A^{-1}}{T} \right] \right\} = 0,
$$
where, for an $N\times N$ symmetric matrix $X$,  we denote by $\lambda_N [X]$ the greatest eigenvalue of $X$. 

In the present work,  for a given target $u_T\in \Lip(\R^N)$ with compact support and a $p$-convex Hamiltonian $H(x,p)$ which may depend also on the space-variable $x$, instead of approximating $u_T$ by the backward-forward method,  we consider the problem of finding a reachable function $u_T^\ast\in \mathcal{R}_T$ that minimizes the $L^2$-distance to $u_T$. In Figure \ref{Fig:L2 projection}, we can see an illustration of the semiconcave envelope of an unreachable target function $u_T$, compared with the $L^2$-projection of $u_T$ onto $\mathcal{R}_T$.
Observe that, since the semiconcave envelope always lies above the target, the projection obtained by the operator $S_T^+\circ S_T^-$ is different to the $L^2$-projection of $u_T$ onto $\mathcal{R}_T$, unless the target $u_T$ is reachable.

\begin{center}
\begin{figure}
\includegraphics[scale=.18]{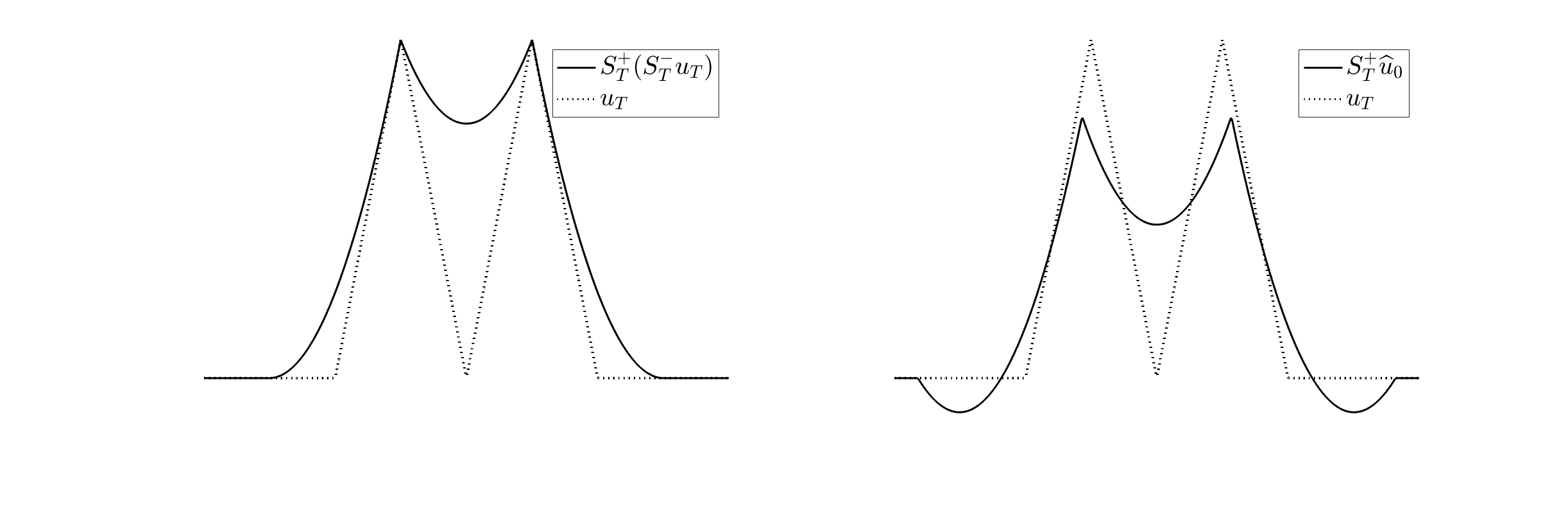}
\caption{At the left we see the semiconcave envelope of a function $u_T\in\Lip(\R)$, obtained by a backward-forward resolution of \eqref{HJ eq}, i.e. the smallest reachable function bounded from below by $u_T$.
At the right we see the $L^2$-projection of $u_T$ onto the reachable set $\mathcal{R}_T$, i.e. the reachable function which minimizes the $L^2$-distance to $u_T$.}
\label{Fig:L2 projection}
\end{figure}
\end{center}

The rest of the paper is structured as follows.
In Section \ref{sec:differentiability}, we present and prove the main results concerning the differentiability of the forward viscosity operator $S_T^+$ with respect to the initial condition. The proof of Theorem \ref{thm: differentiability general weak convergence} is postponed to the subsection \ref{subsec:proof differentiability}.
In subsection \ref{subsec:opt condition},  we compute the gradient of the functional $\mathcal{J}_T$ and derive a first-order optimality condition for the optimal control problem \eqref{L2 OCP}.
In section \ref{sec: linearization},  for the one-dimensional case, and for quadratic Hamiltonians in any space dimension, we identify the directional Gâteaux derivatives obtained in Section \ref{sec:differentiability}, with duality solutions to the transport equation \eqref{eq: transport eq intro Ham}. The main result of this section is stated in subsection \ref{subsec:gateaux deriv duality sol}, and its proof is postponed to subsection \ref{subsec: proof linearization}. In subsection \ref{subsec: opt cond dual}, we represent the gradient of the functional $\mathcal{J}_T$, defined in \eqref{L2 OCP}, by means of the unique backward solution at time $t=0$ to the dual equation to \eqref{eq: transport eq intro Ham}, which is the backward conservative equation \eqref{backward conservative intro}.
In subsection \ref{subsec: duality solutions}, we recall the main definitions and properties of duality solutions, as presented in \cite{bouchut2005uniqueness}. In subsection \ref{subsec:semiconcavity estimates}, we give a proof of a semiconcavity estimate, necessary in the proof of Theorem \ref{thm:gateaux deriv transport}.
Finally, in Section \ref{sec:existence minimizers} we give the proof of the existence of minimizers for the optimal control problem \ref{L2 OCP}. In addition, for the case when the Hamiltonian is $x-$independent and quadratic, we prove that the $L^2$-projection of  any $u_T\in \Lip(\R^N)$ onto the space of reachable targets $\mathcal{R}_T$ is unique.

\section{Differentiability with respect to the initial condition}
\label{sec:differentiability}

The goal in this section is to prove that the forward viscosity operator $S_T^+$ defined in \eqref{forward viscosity operator intro} is differentiable with respect to the $L^1_{loc}$-convergence. Then, we use this differentiability result to compute the gradient of the functional $\mathcal{J}_T$ defined in \eqref{L2 OCP}, which allows us to derive a first-order optimality condition for the optimization problem \eqref{L2 OCP}.

It is well known (see for instance \cite[Chapter 10]{evans1998partial} or \cite[Section I.10]{fleming2006controlled}) that, under the hypotheses \eqref{general Hcond},\eqref{H bounded}, \eqref{H lipschitz in x} and \eqref{H lipschitz in p},  for any initial condition $u_0\in \Lip(\R^N)$, the unique viscosity solution to \eqref{HJ eq} can be given by the value function of an associated problem in the calculus of variations as follows:
\begin{equation}\label{viscosity solution value function}
u(t,x) = 
\inf_{\substack{\xi\in W^{1,1}((0,t);\R^N) \\
\xi(t) = x}}
\left\{
\int_0^t H^\ast (\xi(s),\, \dot{\xi}(s)) ds + u_0 (\xi(0)) \right\}, \qquad \forall (t,x) \in [0,T]\times\R^N
\end{equation}
where $H^\ast: \R^N\times \R^N\to \R$ denotes the Legendre-Fenchel transform of $H(x,\cdot)$, i.e.
$$
H^\ast (x,q) = \max_{p\in \R^N} \left\{ p\cdot q - H(x,p)  \right\} \qquad \forall x,q\in \R^N .
$$
In view of the representation formula \eqref{viscosity solution value function}, for any $t\in (0,T]$, the \emph{forward viscosity operator} $S_t^+$ introduced in \eqref{forward viscosity operator intro} can be
given by the expression
\begin{equation}\label{value function general H}
S_t^+ u_0(x) := 
\inf_{\substack{\xi\in W^{1,1}((0,t);\R^N) \\
\xi(t) = x}}
\left\{
\int_0^t H^\ast (\xi(s),\, \dot{\xi}(s)) ds + u_0 (\xi(0)) \right\}, \qquad \forall x\in \R^N.
\end{equation}
One can readily prove that the operator $S_t^+$ satisfies the semigroup property
\begin{equation*}
(S_{t_1}^+ \circ S_{t_2}^+) u_0 = S_{t_1+t_2}^+ u_0 , \qquad \forall t_1, \, t_2 >0.
\end{equation*}

\subsection{Differentiability of the forward viscosity semigroup}
\label{subsec:diff results}

The main result of this section is the following theorem, which states that the operator $S_t^+$ is G\^ateaux differentiable at any $u_0\in \Lip(\R^N)$, in any direction $w\in \Lip(\R^N)$.
The proof of this theorem is postponed to subsection \ref{subsec:proof differentiability}.

\begin{theorem}\label{thm: differentiability general weak convergence}
Let $T>0$,  $u_0\in \Lip(\R^N)$,  and let $H$ be a Hamiltonian satisfying \eqref{general Hcond},\eqref{H bounded}, \eqref{H lipschitz in x} and \eqref{H lipschitz in p}.
Let $S_t^+$ be the forward viscosity operator defined in \eqref{value function general H}.
Then, for any $w\in \Lip(\R^N)$ and $t\in (0,T]$ we have
$$
\dfrac{S_t^+ (u_0 + \delta w) (\cdot) - S_t^+ u_0 (\cdot) }{\delta} \longrightarrow \partial_w S_t^+ u_0 (\cdot) \quad \text{as} \quad \delta\to 0^+, \quad \text{in} \quad L^1_{loc} (\R^N),
$$
where $\partial_w S_t^+ u_0 (\cdot) \in L^\infty_{loc} (\R^N)$ is defined almost everywhere in $\R^N$ in the following way:

At any point $x\in \R^N$ where $S_t^+ u_0(\cdot)$ is differentiable, we have
\begin{equation}\label{Gateaux derivative thm}
\partial_w S_t^+ u_0 (x) =  w (\xi_{t,x} (0)),
\end{equation}
with $(\xi_{t,x}(\cdot) , p_{t,x} (\cdot)) \in C^1 ([0,t]; \R^N)$ being the unique solution to the backward system of ODEs
\begin{equation}\label{optimality system thm}
\left\{\begin{array}{ll}
\dot{\xi}_{t,x} (s)  = H_p (\xi_{t,x} (s), p_{t,x} (s)) & s\in (0,t) \\
\noalign{\vspace{2pt}}
\dot{p}_{t,x} (s) = -H_x  (\xi_{t,x} (s), p_{t,x} (s)) & s\in (0,t) \\
\noalign{\vspace{2pt}}
\xi_{t,x} (t) = x \quad p_{t,x} (t) = \nabla S_t^+ u_0 (x).
\end{array}\right.
\end{equation}

(Note that,  since $S_t^+u_0(\cdot)\in \Lip(\R^N)$, by means of Rademacher's Theorem we have that $S_t^+u_0 (\cdot)$ is differentiable almost everywhere in $\R^N$.
Hence, the expression \eqref{Gateaux derivative thm} uniquely determines $\partial_w S_t^+ u_0 (\cdot)$ as an element of $L^\infty_{loc} (\R^N)$).

\end{theorem}

The function $ \partial_w S_t^+ u_0(\cdot) \in L^\infty_{loc}(\R^N)$ defined a.e. as in \eqref{Gateaux derivative thm} can be interpreted as the G\^ateaux derivative of the operator $S_t^+$ at $u_0$ in the direction $w$. 
Now, for any $u_0\in \Lip(\R^N)$ and $t\in (0,T]$ fixed, let us define the function $\Phi_{u_0}^t (\cdot)\in L_{loc}^\infty (\R^N;\R^N)$ as
\begin{equation}\label{Phi def}
\Phi_{u_0}^t (x):= \xi_{t,x} (0), \qquad \text{for a.e.} \ x\in\R^N,
\end{equation}
where $\xi_{t,x} (0)$ is the unique solution to \eqref{optimality system thm}. Then, the operator
$$
\begin{array}{ccclc}
DS_t^+(u_0): & \Lip(\R^N) & \longrightarrow & L^\infty_{loc} (\R^N) \\
\noalign{\vspace{2pt}}
& w(\cdot) & \longmapsto & \partial_w S_t^+ u_0 (\cdot) := w(\Phi_{u_0}^t(\cdot)) & \text{a.e.  in} \ \R^N
\end{array}
$$
is linear and continuous  in its domain, and can be seen as the gradient of the operator $S_t^+$ at $u_0$.
This proves that for any $t\in (0,T]$, the operator $S_t^+$ defined in \eqref{value function general H} is differentiable with respect to the $L^1_{loc}$-convergence.

\begin{remark}[Lipschitz continuity]\label{rmk: S_T^+ Lipschitz}
Note moreover that, for any $u_0\in\Lip(\R^N)$, the norm of the linear operator $DS_t^+(u_0)$ 
above defined, restricted to $\Lip(\R^N)\cap L^\infty(\R^N)$, is equal to $1$. This implies that, for any $t\in (0,T]$, the forward viscosity operator $S_t^+$, restricted to $\Lip(\R^N)\cap L^\infty(\R^N)$, is Lipschitz with constant $1$, i.e.
$$
\| S_t^+ u_0 - S_t^+ u_1 \|_{L^\infty(\R^N)} \leq \| u_0 - u_1 \|_{L^\infty (\R^N)} \quad \text{for all} \quad u_0,u_1\in \Lip(\R^N)\cap L^\infty(\R^N).
$$
This inequality is also proved  in \eqref{lebesgue 2} as part of the proof of Theorem \ref{thm: differentiability general weak convergence}.
\end{remark}

\subsection{First-order optimality condition for inverse designs}
\label{subsec:opt condition}

In this subsection, for a given time horizon $T>0$ and a target function $u_T\in \Lip_0(\R^N)$, we consider the optimization problem \eqref{L2 OCP}, presented in the introduction, of constructing an initial condition for which the corresponding viscosity solution to \eqref{HJ eq} at time $T$ minimizes the $L^2$-distance to $u_T$.
Using the forward viscosity operator $S_T^+$ introduced in \eqref{forward viscosity operator intro}, and given explicitly by the expression \eqref{value function general H}, the problem can be formulated as
\begin{equation}\label{L2 OCP S^+}
\underset{u_0\in \Lip_0(\R^N)}{\text{minimize}} \  \mathcal{J}_T (u_0) : =  \| S_T^+ u_0 (\cdot) - u_T(\cdot) \|_{L^2 (\R^N)}^2 .
\end{equation}

\begin{remark}[Assumptions on the Hamiltonian]
\label{rmk: C=0}
We address the optimization problem \eqref{L2 OCP S^+} under the assumption that the Hamiltonian $H$ satisfies the hypotheses \eqref{general Hcond}, \eqref{H lipschitz in x}, \eqref{H lipschitz in p}, and \eqref{H bounded} with $C_0=0$. This choice guarantees that the operator $S_T^+$ defined in \eqref{value function general H} satisfies
\begin{equation}\label{lip_0 to L2}
S_T^+   u_0 \in \Lip_0(\R^N) \qquad \forall u_0\in \Lip_0(\R^N).
\end{equation}
Indeed, it is easy to check that condition \eqref{H bounded} with $C_0=0$ implies the same property on the Legendre-Fenchel transform $H^\ast$, i.e.
$$
H^\ast (x,q) \geq 0 \quad \forall x,q\in \R^N \quad \text{and} \quad
H^\ast (x,0) = 0 \quad  \forall x\in \R^N.
$$
This property implies in particular that $S_T^+ 0 = 0$,
and then, by using the compactness estimates from \cite[Corollary 1]{ancona2016quantitative} (see also \cite{ancona2016compactness}), we deduce that, if $u_0 \in \Lip_0 (\R^N)$, there exists a constant $l>0$ such that $\operatorname{supp}(S_T^+u_0) \subset [-l,l]^N$. Finally, since $\Lip_0 (\R^N)\subset L^2 (\R^N)$, property \eqref{lip_0 to L2} follows.
Having property \eqref{lip_0 to L2} in hand, we deduce that the set of admissible initial conditions for the optimization problem \eqref{L2 OCP S^+}, i.e. $u_0\in \Lip(\R^N)$ satisfying $\mathcal{J}_T (u_0)< \infty$, is nonempty.
\end{remark}

\begin{remark}[Existence of minimizers]
Note that, due to the lack of backward uniqueness of the Hamilton-Jacobi equation \eqref{HJ eq}, the functional $\mathcal{J}_T(\cdot)$ is not expected to be coercive. Indeed, in view of the results in \cite{colombo2019initial,esteve2020inverse}, the preimage by $S_T^+$ of any reachable function $\varphi_T\in \mathcal{R}_T$ consists of a convex cone in $\Lip(\R^N)$,  which might be unbounded in $L^2(\R^N)$.
The existence of a minimizer for the problem \eqref{L2 OCP S^+} can however be justified by using the backward viscosity operator $S_T^-$ defined in \eqref{value function backward general H}. 
We shall prove, in Theorem \ref{thm:existence L2 projection} in Section \ref{sec:existence minimizers}, the existence of at least one minimizer for the optimization problem \eqref{L2 OCP S^+}.
\end{remark}

In the following theorem,  we derive a first-order optimality condition for the problem \eqref{L2 OCP S^+}. 
This is done by using the differentiability result from Theorem \ref{thm: differentiability general weak convergence}, which allows one to compute the directional derivatives of the functional
$$
\begin{array}{cccc}
\mathcal{J}_T: & \Lip_0(\R^N) & \longrightarrow & \R \\
\noalign{\vspace{2pt}}
 &  u_0 & \longmapsto & \| S_T^+ u_0 (\cdot) - u_T(\cdot)\|_{L^2(\R^N)}^2.
\end{array}
$$

\begin{theorem}\label{thm:optimality condition}
Let $T>0$ and let $H$ be a Hamiltonian satisfying \eqref{general Hcond},  \eqref{H lipschitz in x}, \eqref{H lipschitz in p}, and  \eqref{H bounded} with $C_0=0$.
Let $S_T^+$ be the forward viscosity operator defined in \eqref{value function general H},
and let $u_T\in \Lip_0(\R^N)$ be a given function with compact support.
Then,  any $u_0\in \Lip_0(\R^N)$ solution to the optimization problem \eqref{L2 OCP S^+}  satisfies
\begin{equation}\label{1st order opt cond}
\displaystyle\int_{\R^N} (S_T^+ u_0 (x) - u_T(x)) \,  w(\Phi_{u_0}^T(x))dx = 0 \quad \forall w\in \Lip(\R^N).
\end{equation}

This condition can also be expressed, in terms of the gradient of the functional $\mathcal{J}_T$,  as 
$$D \mathcal{J}_T (u_0) = 0,$$
where $D\mathcal{J}_T$ is the continuous linear functional
\begin{equation}\label{gradient J_T}
\begin{array}{ccccl}
D\mathcal{J}_T(u_0) : & \hspace{-5pt}  \Lip(\R^N) & \rightarrow & \R \\
\noalign{\vspace{2pt}}
& w(\cdot ) & \mapsto & \partial_w \mathcal{J}_T (u_0) & \hspace{-8pt} : = 2 \displaystyle\int_{\R^N} (S_T^+ u_0 (x) - u_T(x)) \,  w(\Phi_{u_0}^T(x))dx
\end{array}
\end{equation}
with $\Phi_{u_0}^T:\R^N\to \R^N$ being the map defined a.e. as in \eqref{Phi def}.
Note that, by the density of Lipschitz functions in $C_c^0(\R^N)$, we can extend the functional $D\mathcal{J}_T (u_0)$ to $C_c^0(\R^N)$, and we then deduce that $D\mathcal{J}_T (u_0)$  defines a Radon measure in $\R^N$.
\end{theorem}

\begin{proof}
Using the conclusions of Theorem \ref{thm: differentiability general weak convergence}, and the fact that $u_0,u_T\in \Lip_0(\R^N)$ are compactly supported, we can compute
\begin{eqnarray*}
\lim_{\delta\to 0^+} \dfrac{\mathcal{J}_T(u_0 + \delta w)- \mathcal{J}_T (u_0)}{\delta} &=& \lim_{\delta \to 0^+} 2\displaystyle\int_{\R^N} (S_T^+ u_0 (x) - u_T(x)) \dfrac{S_T^+(u_0 + \delta w)(x) - S_T^+ u_0(x)}{\delta} dx \\
 &=& 2\displaystyle\int_{\R^N} (S_T^+ u_0 (x) - u_T(x)) \partial_w S_T^+ u_0(x) dx,
\end{eqnarray*}
where $\partial_w S_T^+ u_0\in L^\infty_{loc} (\R^N)$ is given by \eqref{Gateaux derivative thm}. 
Hence, using the map $\Phi_{u_0}^T$ defined a.e. in \eqref{Phi def}, we can write the directional derivative of $\mathcal{J}_T$ at $u_0$ in the direction $w$ as
$$
\partial_w \mathcal{J}_T (u_0) = 2 \int_{\R^N} (S_T^+ u_0 (x) - u_T(x)) \,  w(\Phi_{u_0}^T(x))dx.
$$
This gives the first-order optimality condition \eqref{1st order opt cond}.

Finally, note that the functional $D \mathcal{J}_T$ defined in \eqref{gradient J_T} is clearly linear and continuous with respect to the sup-norm in $C_c^0(\R^N)$.   Then, by the density of $\Lip(\R^N)$ in the space of continuous functions with compact support, we can identify $D\mathcal{J}_T(u_0)$ with a Radon measure in $\R^N$.
\end{proof}

\subsection{Proof of Theorem \ref{thm: differentiability general weak convergence}}
\label{subsec:proof differentiability}

Let us give the proof of Theorem \ref{thm: differentiability general weak convergence}.
As a first step we provide, in Proposition \ref{prop: differentiability general}, an explicit expression for the limit
$$
 \lim_{\delta \to 0^+} \dfrac{S_t^+ (u_0 + \delta w)(x) - S_t^+ u_0(x)}{\delta}.
$$
at the points $x\in \R^N$ where $S_t^+ u_0$ is differentiable.
As it is well-known, the viscosity solution to \eqref{HJ eq} is Lipschitz continuous, and then, by Rademacher's Theorem,  $S_t^+u_0$ is differentiable for a.e.  $x\in\R^N$.
Therefore, Proposition \ref{prop: differentiability general} is sufficient to explicitly identify a unique candidate in  $L^\infty_{loc} (\R^N)$ for the directional G\^ateaux derivative $\partial_w S_t^+ u_0(\cdot)$.

\begin{proposition}\label{prop: differentiability general}
Let $T>0$, $u_0\in \Lip(\R^N)$ and $H$ be a Hamiltonian satisfying \eqref{general Hcond},\eqref{H bounded}, \eqref{H lipschitz in x} and \eqref{H lipschitz in p}.
For any $t\in (0,T]$, let $S_t^+$ be the forward viscosity operator defined in \eqref{value function general H}.
For any $x\in\R^N$ such that $S_t^+u_0(\cdot)$ is differentiable at $x$,
let $\xi_{t,x} \in C^1([0,t]; \R^N)$,  together with $p_{t,x} \in C^1([0,t]; \R^N)$, be the unique solution to the terminal value problem
\begin{equation}\label{optimality system}
\left\{\begin{array}{ll}
\dot{\xi}_{t,x} (s)  = H_p (\xi_{t,x} (s), p_{t,x}(s)) & s\in (0,t) \\
\noalign{\vspace{2pt}}
\dot{p}_{t,x} (s) = -H_x  (\xi_{t,x} (s), p_{t,x}(s)) & s\in (0,t) \\
\noalign{\vspace{2pt}}
\xi_{t,x} (t) = x \quad p_{t,x}(t) = \nabla S_t^+ u_0 (x).
\end{array}\right.
\end{equation}
Then,  for any $w\in \Lip (\R^N)$, we have
$$
 \lim_{\delta \to 0^+} \dfrac{S_t^+ (u_0 + \delta w)(x) - S_t^+ u_0(x)}{\delta} = w (\xi_{t,x}(0)).
$$
\end{proposition}

\begin{remark}
This result proves that, for any  $x\in \R^N$ fixed, the map
$$
u_0(\cdot)\in \Lip(\R^N) \longmapsto S_t^+ u_0 (x)\in \R
$$
is differentiable at any $u_0\in \Lip(\R^N)$ where $S_t^+ u_0(\cdot)$ is differentiable. However, it does not directly prove the differentiability of the operator $S_t^+$.
Indeed, Proposition \ref{prop: differentiability general}  proves that the limit
$$
\dfrac{S_t^+ (u_0 + \delta w)(\cdot) - S_t^+ u_0(\cdot)}{\delta} \to \partial_w S_t^+ u_0 (\cdot) \qquad \text{as} \ \delta \to 0^+.
$$
holds in the sense almost everywhere in $\R^N$,  and represents the first step of the proof of Theorem \ref{thm: differentiability general weak convergence}.
The second step will consist in justifying that the convergence actually holds in $L^1_{loc}(\R^N)$.
\end{remark}

In the proof of Proposition \ref{prop: differentiability general}, we use the following well-known result concerning the minimizers of the right-hand-side of  \eqref{value function general H} at the points where $S_t^+ u_0 (\cdot)$ is differentiable.

\begin{lemma}\label{lem:regular points}
Let $T>0$,  $u_0\in \Lip (\R^N)$ and $H$ be a Hamiltonian satisfying \eqref{general Hcond},\eqref{H bounded}, \eqref{H lipschitz in x} and \eqref{H lipschitz in p}.
For any $(t,x) \in (0,T]\times \R^N$ fixed,   let $\xi_{t,x}(\cdot)$ be any minimizing trajectory for the right-hand side of  \eqref{value function general H}, and let $p_{t,x}(\cdot)$ be its associated dual arc. 
Then, the following statements hold true:
\begin{enumerate}
\item For all $s\in (0,t)$, the function $S_s^+ u_0(\cdot)$ is differentiable at $\xi_{t,x}(s)$, and 
$$\nabla S_s^+ u_0 (\xi_{t,x}(s)) = p_{t,x}(s).$$
\item If  $S_t^+ u_0(\cdot)$ is differentiable at $x$, then the minimizer $\xi_{t,x}$ of the right-hand-side in \eqref{value function general H} is unique and is given by the unique solution to \eqref{optimality system}.
\end{enumerate}
\end{lemma}

\begin{proof}
The result is well-known and is a direct consequence of the results in Chapter 6 in \cite{cannarsa2004semiconcave}.
Let us give a brief sketch of the proof for the reader's convenience.
The first statement follows directly from \cite[Theorem 6.4.8]{cannarsa2004semiconcave}.
For the second statement, sing \cite[Theorem 6.4.10]{cannarsa2004semiconcave},  the differentiability of $S_t^+ u_0 (\cdot)$ at $x$ implies that $(t,x)$ is a regular point in the sense of \cite[Definition 6.3.4]{cannarsa2004semiconcave}, i.e. the minimizer $\xi_x$ of the right-hand-side in \eqref{value function general H} is unique.
Then the proof can be concluded by using Theorems 6.3.3 and 6.4.8 in \cite{cannarsa2004semiconcave}.
\end{proof}

We can now proceed to the proof of Proposition \ref{prop: differentiability general}.

\begin{proof}[Proof of Proposition \ref{prop: differentiability general}]
Let $(t,x)\in (0,T]\times \R^N$ be such that $S_t^+u_0 (\cdot)$ is differentiable at $x$, and let $\xi_{t,x}(\cdot)\in C^1([0,t];\R^N)$ be the solution to \eqref{optimality system}.
We first note that, in view of \eqref{value function general H} and Lemma \ref{lem:regular points}, we have
$$
S_t^+ u_0 (x) = \int_0^t H^\ast (\xi_{t,x} (s), \dot{\xi}_{t,x} (s)) ds + u_0 (\xi_{t,x} (0))
$$
and
$$
S_t^+ (u_0 + \delta w) (x) \leq  \int_0^t H^\ast (\xi_{t,x} (s), \dot{\xi}_{t,x} (s)) ds + u_0 (\xi_{t,x} (0)) + \delta w (\xi_{t,x}(0)),
$$
which combined together imply
$$
\dfrac{S_t^+ (u_0 +\delta w) (x) - S_t^+ u_0 (x)}{\delta} \leq w(\xi_{t,x}(0)) \qquad \forall \delta>0.
$$
Hence, we have
\begin{equation}\label{limsup proof differentiability}
\limsup_{\delta \to 0^+} \dfrac{S_t^+ (u_0 +\delta w) (x) - S_t^+ u_0 (x)}{\delta} \leq w(\xi_{t,x}(0)).
\end{equation}

Let us now prove that
\begin{equation}\label{liminf proof differentiability}
\liminf_{\delta \to 0^+} \dfrac{S_t^+ (u_0 +\delta w) (x) - S_t^+ u_0 (x)}{\delta} \geq w(\xi_{t,x}(0)).
\end{equation}

We argue by contradiction. Let us suppose that there exists $\alpha>0$ and a sequence $\{ \delta_n\}_{n\in \N}$ such that $\delta_n \in (0,1)$,  $\delta_n\to 0$ and
\begin{equation}\label{contradiction hyp}
\dfrac{S_t^+ (u_0 + \delta_n w)(x) - S_t^+ u_0(x)}{\delta_n} \leq w(\xi_{t,x}(0)) - \alpha \qquad \forall n\in\N.
\end{equation}

For each $n\in\N$, let $\xi_n \in C^1 (0,t;\R^N)$ be such that
$$
S_t^+ (u_0 + \delta_n w)(x) = \int_0^t H^\ast (\xi_n (s), \dot{\xi}_n (s)) ds + u_0 (\xi_n (0)) + \delta_n w(\xi_n(0)).
$$
Using \eqref{value function general H}, we deduce
\begin{eqnarray*}
S_t^+ u_0 (x) + \delta_n w(\xi_n (0)) & = & \int_0^t H^\ast (\xi_{t,x}(s), \dot{\xi}_{t,x} (s) ) ds + u_0 (\xi_{t,x} (0)) + \delta_n w (\xi_n (0)) \\
&\leq &  \int_0^t H^\ast (\xi_n(s), \dot{\xi}_n (s) ) ds + u_0 (\xi_n (0)) + \delta_n w (\xi_n (0)) \\
&=& S_t^+ (u_0 +\delta_n w)(x),
\end{eqnarray*}
which combined with \eqref{contradiction hyp} implies
\begin{equation}\label{xi n far from optimal}
w (\xi_n(0)) \leq \dfrac{S_t^+ (u_0 + \delta_n w) (x) - S_t^+ u_0 (x)}{\delta_n} \leq w(\xi_{t,x}(0)) - \alpha \qquad \forall n\in\N.
\end{equation}

Using the fact that the function $u_0 + \delta_n w$ is globally Lipschitz with a Lipschitz constant independent of $n$, we deduce from Lemma \ref{lem:regular points}, and using the hypothesis \eqref{H lipschitz in x}, that there exists a constant $C_T >0$ such that the viscosity solution $u_n \in \Lip([0,T]\times\R^N)$
to \eqref{HJ eq} with initial condition $u_0+\delta_n w$ is Lipschitz continuous with Lipschitz constant $C_T$ independent of $n\in \N$.
This implies in particular that the dual arc $p_n(\cdot)$ associated to each optimal trajectory $\xi_n(\cdot)$ satisfies
\begin{equation*}
|p_n(s)| \leq C_T  \qquad \forall s\in [0,t].
\end{equation*}
We then deduce, using \eqref{H lipschitz in p}, that there exists another constant $C_T'>0$, also independent of $n$, such that
$$
| \dot{\xi}_n (s) | \leq C_T' \qquad \forall s\in [0,t] \quad \forall n\in \N.
$$
Hence,  by Arzéla-Ascoli Theorem, there exists $\xi_\infty(\cdot) \in W^{1,1}(0,T;\R^N)$ such that $\xi_n(\cdot) \to \xi_\infty(\cdot)$ uniformly in $(0,t)$.

Now,  in view of \eqref{xi n far from optimal}, we deduce that $\xi_\infty (\cdot) \neq \xi_{t,x}(\cdot)$,
which combined with
 the fact that, by means of Lemma \ref{lem:regular points} (ii), $\xi_{t,x}(\cdot)$ is the unique minimizer for the right-hand-side of \eqref{value function general H}, implies that
there exists $\beta >0$ such that
$$
\int_0^t H^\ast (\xi_\infty (s), \dot{\xi}_\infty (s) )ds + u_0 (\xi_\infty (0)) = 
\int_0^t H^\ast (\xi_{t,x} (s), \dot{\xi}_{t,x} (s) )ds + u_0 (\xi_{t,x} (0)) + \beta.
$$

On the other hand, since $H^\ast (x,q)$ is continuous with respect to both variables and convex in $q$, we deduce that the right-hand-side of \eqref{value function general H} is weakly lower semicontinuous, and then
\begin{eqnarray*}
\liminf_{n\to \infty} \int_0^t H^\ast (\xi_n (s), \dot{\xi}_n (s) )ds + u_0 (\xi_n (0)) & \geq & \int_0^t H^\ast (\xi_\infty (s), \dot{\xi}_\infty (s) )ds + u_0 (\xi_\infty (0)) \\
&=& \int_0^t H^\ast (\xi_{t,x} (s), \dot{\xi}_{t,x} (s) )ds + u_0 (\xi_{t,x} (0)) + \beta,
\end{eqnarray*}
which implies that we can extract another subsequence $\delta_n \to 0$ such that
$$
 \int_0^t H^\ast (\xi_n (s), \dot{\xi}_n (s) )ds + u_0 (\xi_n (0))
 \geq \int_0^t H^\ast (\xi_{t,x} (s), \dot{\xi}_{t,x} (s) )ds + u_0 (\xi_{t,x} (0)) + \dfrac{\beta}{2} \qquad \forall n\geq 1.
 $$
 
Finally, using the optimality of $\xi_n$ for each $n$, we have in addition that
 \begin{eqnarray*}
\displaystyle\int_0^t H^\ast (\xi_{t,x} (s), \dot{\xi}_{t,x} (s) )ds + u_0 (\xi_{t,x} (0)) + \delta_n w(\xi_{t,x}(0)) & \geq &
 \displaystyle\int_0^t H^\ast (\xi_n (s), \dot{\xi}_n (s) )ds + u_0 (\xi_n (0)) \\
 & &  + \delta_n w(\xi_n(0))
 \end{eqnarray*}
for all $n\geq 1$, and adding the last two inequalities we obtain
 $$
 \delta_n \left( w(\xi_{t,x}(0)) - w(\xi_n (0))  \right) \geq \dfrac{\beta}{2} >0,
 $$
 which leads to a contradiction since $\delta_n\to 0$.
 The inequality \eqref{liminf proof differentiability} then follows, and together with \eqref{limsup proof differentiability}, implies the conclusion of the theorem.
\end{proof}

Let us finish the section with the proof of Theorem \ref{thm: differentiability general weak convergence}.

\begin{proof}[Proof of Theorem \ref{thm: differentiability general weak convergence}]
In one hand, since $S_t^+u_0(\cdot)$ is Lipschitz continuous, and then differentiable almost everywhere in $\R^N$, by means of Proposition \ref{prop: differentiability general} we have
\begin{equation}\label{lebesgue 1}
\dfrac{S_t^+ (u_0 + \delta w) (x) - S_t^+ u_0 (x) }{\delta}
\to  \partial_w S_t^+ u_0 (x) \quad \text{for a.e. } \ x\in \R^N.
\end{equation}
Now, let $K\subset\subset \R^N$ be any compact set.
We claim that there exists another compact set $K'\subset\subset \R^N$ such that, for any $\delta>0$, it holds that
\begin{equation}\label{lebesgue 2}
\left| \dfrac{S_t^+ (u_0 + \delta w) (x) - S_t^+ u_0 (x)}{\delta}  \right| \leq \| w\|_{L^\infty (K')} \qquad \forall x\in \R^N.
\end{equation}
Indeed, by the definition of $S_t^+$ in \eqref{value function general H},  for any $x\in K$ there exists $\xi_1 \in C^1 (0,t;\R^N)$ such that
$$
S_t^+ u_0 (x) = \int_0^t H^\ast (\xi_1 (s), \dot{\xi}_1 (s)) ds + u_0 (\xi_1 (0)),
$$
which then implies that
\begin{eqnarray}
S_t^+ (u_0 + \delta w) (x) - S_t^+ u_0 (x) & \leq &  \int_0^t H^\ast (\xi_1 (s), \dot{\xi}_1 (s)) ds + u_0 (\xi_1 (0))+ \delta w (\xi_1(0)) \nonumber \\
& &  -  \int_0^t H^\ast (\xi_1 (s), \dot{\xi}_1 (s)) ds - u_0 (\xi_1 (0)) \label{estimate compact set}
 \\
&=& \delta w (\xi_1(0)). \nonumber
\end{eqnarray}
Moreover, using the hypotheses \eqref{H lipschitz in x} and \eqref{H lipschitz in p}, we deduce that the trajectory $\xi_1(\cdot)$ is contained in a compact set $K'\subset\subset \R^N$ independent of $x\in K$,
which in turn implies,  combined with \eqref{estimate compact set}, that
$$
S_t^+ (u_0 + \delta w) (x) - S_t^+ u_0 (x) \leq 
 \delta \| w\|_{L^\infty(K')}.
$$
Using a similar argument, with an arc $\xi_2 \in C^1 (0,t;\R^N)$ minimizing the right-hand-side of \eqref{value function general H} with $u_0+\delta w$ instead of $u_0$ we obtain
$$
S_t^+ (u_0 + \delta w) (x) - S_t^+ u_0 (x) \geq  \delta w (\xi_2(0)) \geq - \delta \| w\|_{L^\infty (K')},
$$
and then \eqref{lebesgue 2} follows.

Now,  in view of  \eqref{lebesgue 1} and \eqref{lebesgue 2}, we can use Lebesgue's dominated convergence Theorem to deduce
$$
\lim_{\delta \to 0^+} \int_{K}\left| \dfrac{S_t^+ (u_0 + \delta w) (x) - S_t^+ u_0 (x) }{\delta} -  \partial_w S_t^+ u_0 (x) \right| dx = 0, \quad \text{for all compact set $K\subset\subset \R^N$,}
$$
which implies that 
$$
\dfrac{S_t^+ (u_0 + \delta w)  - S_t^+ u_0  }{\delta} 
\longrightarrow  \partial_w S_t^+ u_0 , \qquad \text{as} \quad  \delta\to 0^+ \quad \text{in} \ L^1_{loc}(\R^N).
$$
\end{proof}

\section{G\^ateaux derivatives and linear transport equations}
\label{sec: linearization}

The goal in this section is to prove that, for any $u_0, w \in \Lip (\R^N)$,  the directional Gateaux derivative $\partial_w S_T^+ u_0$, that we proved to exist in Theorem \ref{thm: differentiability general weak convergence},  is actually the unique duality solution to the following linear transport equation:
\begin{equation}\label{transport equation intro}
\left\{\begin{array}{ll}
\partial_t v + a(t,x)\cdot  \nabla_x v = 0 & (t,x) \in (0,T)\times \R^N \\
v(0,x) = w(x) & x\in \R^N,
\end{array}\right.
\end{equation}
where the transport coefficient $a(t,x)$ is given by
\begin{equation}\label{transport-coef thm}
a(t,x) := H_p (x,\nabla_x [S_t^+ u_0 (x)]).
\end{equation}
The precise definition and the main properties of duality solutions for the equation \eqref{transport equation intro} are given in subsection \ref{subsec: duality solutions}. We refer to \cite{bouchut1998one,bouchut2005uniqueness} for further details concerning the theory of duality solutions for transport equations with discontinuous transport coefficient satisfying a one-sided-Lipschitz condition.

As it is well-known, when $t>0$ is sufficiently large, the solution to \eqref{HJ eq} eventually looses regularity and its gradient develops jump-discontinuities.
This in turn implies that the transport coefficient \eqref{transport-coef thm} is no longer continuous, and therefore, the  equation \eqref{transport equation intro} cannot be solved by the classical method of characteristics.
A key-feature in our proof to establish well-posedness for \eqref{transport equation intro} is the fact that the transport coefficient $a(t,x)$ defined in \eqref{transport-coef thm} satisfies the following one-sided Lipschitz condition (OSLC)
\begin{equation}\label{OSLC OSLC}
\langle a(t,y)-a(t,x), \, y-x \rangle \leq \dfrac{C}{t}  |y-x|^2  \qquad \text{for a.e.} \  (t,x,y)\in (0,T)\times\R^N\times\R^N,
\end{equation}
where $C>0$ is a positive constant.
In the work by Bouchut-James-Mancini \cite{bouchut2005uniqueness}, existence, uniqueness and stability is established for linear transport equations as \eqref{transport equation intro} with a transport coefficient satisfying the OSLC condition. However, we are not able to directly apply the results in \cite{bouchut2005uniqueness} since the function $\alpha(t) = C/t$ does not belong to $L^1(0,T)$.
Moreover, we are only able to treat the one-dimensional case in space, and the case of quadratic Hamiltonians of the form
\begin{equation}
\label{quadratic Hamiltonian section}
H(x,p) = |p|^2 + f(x),
\end{equation}
where $f\in C^2(\R^N)$ is bounded and globally Lipschitz.
In these two cases, \eqref{OSLC OSLC} can be deduced from the semiconcavity of the viscosity solution, whereas this is not the case for general convex Hamiltonians in dimension higher than 1.

\subsection{Directional Gâteaux derivatives as duality solutions}
\label{subsec:gateaux deriv duality sol}

Let us state the main result of this section, which ensures that the directional Gâteaux derivative of the forward viscosity operator $S_T^+$ at $u_0$ in the direction $w$ is the unique duality solution to the linear transport equation \eqref{transport equation intro}.

\begin{theorem}\label{thm:gateaux deriv transport}
Let $T>0$, $u_0, w \in \Lip(\R^N)$, and let, either $N=1$ and $H$ be a Hamiltonian satisfying \eqref{general Hcond}, \eqref{H bounded}, \eqref{H lipschitz in x} and \eqref{H lipschitz in p}, or let $N>1$ and $H$ be of the form \eqref{quadratic Hamiltonian section}. Then we have
$$
\dfrac{S_t^+ (u_0 + \delta w)(x) - S_t^+ u_0 (x)}{\delta} \to v(t,x)  \qquad \text{in} \ C([0,T], L^1_{loc} (\R^N))
$$
where $v(t,x)$ is the unique duality solution to the linear transport equation
\eqref{transport equation intro} with transport coefficient $a(t,x)$ given by \eqref{transport-coef thm} and initial condition $v(0,\cdot) = w(\cdot)$.
\end{theorem}

The proof of this theorem is postponed to subsection \ref{subsec: proof linearization},
and relies on the well-posedness of the dual equation to \eqref{transport equation intro},  which is the backward conservative equation
 \begin{equation}\label{eq: conservative backward opt cond}
\left\{
\begin{array}{ll}
\partial_t \pi + \operatorname{div}_x (a(t,x) \pi) = 0 & \text{in} \ (0,T)\times \R^N, \\
\noalign{\vspace{2pt}}
\pi (T,x) = \pi^T (x) & \text{in} \ \R^N.
\end{array}\right.
\end{equation}
where the coefficient $a(t,x)$ is given by \eqref{transport-coef thm}, and $\pi^T\in L^\infty(\R^N)$ is any given terminal condition.
Using the results in \cite{bouchut2005uniqueness}, we can deduce that there exists a unique reversible\footnote{See Definition \ref{def: reversible sol} and Theorem \ref{thm:reversible wellposedness}.} solution $\pi \in C((0,T], L^\infty-w^\ast)$ to the terminal value problem \eqref{eq: conservative backward opt cond}.
However, the fact that the right-hand-side of \eqref{OSLC OSLC} does not belong to $L^1(0,T)$, prevents us from extending by continuity\footnote{with respect to the $\text{weak}^\ast$ topology in $L^\infty(\R^N)$.} $\pi$ at $t=0$ by a $L^\infty$ function.
This allows us to prove, in Proposition \ref{prop: deriv transport eq}, that the Gâteaux derivative $v(t,x) = \partial_w S_T^+u_0$ is a duality solution to the forward equation \eqref{transport equation intro}.
However, in order to ensure that this duality solution is actually the unique duality solution, we need to use Proposition \ref{prop:reversible unique extension}, where we prove that the reversible solutions to \eqref{eq: conservative backward opt cond} can be uniquely extended by continuity at $t=0$ by a Radon measure in $\R^N$.
This unique extension result uses in a crucial manner the fact that the coefficient $a(t,x)$ is of the form \eqref{transport-coef thm}.

\subsection{First-order optimality condition by means of the dual equation}\label{subsec: opt cond dual}

Here, we  use the differentiability result from Theorem \ref{thm:gateaux deriv transport} to compute the gradient of the functional $\mathcal{J}_T$ defined in the
 optimization problem \eqref{L2 OCP S^+}, that we recall here
$$
\underset{u_0\in \Lip_0(\R^N)}{\text{minimize}} \  \mathcal{J}_T (u_0) : =  \| S_T^+ u_0 (\cdot) - u_T(\cdot) \|_{L^2 (\R^N)}^2 ,
$$
 in terms of the dual equation to the linear transport equation \eqref{transport equation intro}, which is the backward conservative equation \eqref{eq: conservative backward opt cond}, which will be proved to be well-posed in Proposition \ref{prop:reversible unique extension},  in the class of reversible solutions (see \cite{bouchut1998one, bouchut2005uniqueness}), for any terminal condition $\pi^T\in L^\infty(\R^N)$ with compact support.

\begin{theorem}
\label{thm: opt cond dual equation}
Let $T>0$, $u_0, w \in \Lip(\R^N)$, and let, either $N=1$ and $H$ be a Hamiltonian satisfying \eqref{general Hcond}, \eqref{H bounded}, \eqref{H lipschitz in x} and \eqref{H lipschitz in p}, or let $N>1$ and $H(p) = |p|^2$.
Then, the gradient of the functional $\mathcal{J}_T$ is given by the linear functional
\begin{equation}
\label{grad J_T duality}
\begin{array}{cccc}
D \mathcal{J}_T (u_0) : & \Lip(\R^N) & \longrightarrow & \R^N \\
& w(\cdot) & \longmapsto & \partial_w \mathcal{J}_T (u_0) := 2 \displaystyle\int_{\R^N} w(x) d\pi^0(x),
\end{array}
\end{equation}
where $\pi^0 \in \mathcal{M}(\R^N)$ is the unique Radon measure which extends\footnote{The uniqueness of the extension at $t=0$ is shown in Proposition \ref{prop:reversible unique extension}.} by continuity in $\mathcal{M}(\R^N)-w^\ast$, at time $t=0$, the unique reversible solution to the conservative transport equation \eqref{eq: conservative backward opt cond} with terminal condition
$$
\pi^T(x) = S_T^+ u_0 (x) - u_T(x) \qquad \forall x\in \R^N.
$$ 
Hence, any initial condition $u_0\in \Lip_0(\R^N)$, solution to the optimization problem \eqref{L2 OCP S^+} ,  satisfies  $D \mathcal{J}_T (u_0) =0,$ where $D\mathcal{J}_T(u_0)$ is the Radon measure defined in \eqref{grad J_T duality}.
\end{theorem}

\begin{proof}
We prove the Theorem \ref{thm: opt cond dual equation}, assuming that Theorem \ref{thm:gateaux deriv transport} (that we will prove later) is true.
We need to prove that the linear functional $D\mathcal{J}_T (u_0)$ defined in \eqref{grad J_T duality} satisfies
$$
D \mathcal{J}_T(u_0) [w]  = \partial_w \mathcal{J}_T(u_0) :=  \lim_{\delta\to 0^+} \dfrac{\mathcal{J}_T (u_0 + \delta w) - \mathcal{J}_T (u_0)}{\delta}, \qquad \forall w\in \Lip(\R^N).
$$

Using the definition of $\mathcal{J}_T$ and Theorem \ref{thm:optimality condition}, together with Theorem \ref{thm:gateaux deriv transport} and the fact that $S_T^+ u_0 - u_T$ is compactly supported, we can compute
\begin{eqnarray*}
 \partial_w \mathcal{J}_T(u_0)  &=& 2 \displaystyle\int_{\R^N} \left( S_T^+ u_0 (x) - u_T(x) \right) \partial_w S_T^+ u_0 (x) dx \\
&=& 2 \displaystyle\int_{\R^N}  \left( S_T^+ u_0 (x) - u_T(x) \right) v(T,x) dx,
\end{eqnarray*}
where $v(T,\cdot)$ is the unique duality solution to \eqref{eq: conservative backward opt cond} at time $t=T$, with initial condition $w$. 
Now, using the definition of duality solution (see Definition \ref{def:duality sol}), we have that the map
$$
t\longmapsto \displaystyle\int_{\R^N} v(t,x) \pi(t,x) dx \qquad \text{is constant in} (0,T],
$$
where $\pi \in C((0,T], L^\infty-w^\ast)$ is the unique reversible solution to the conservative transport equation \eqref{eq: conservative backward opt cond} with terminal consition $\pi^T(\cdot) = S_T^+u_0(\cdot) - u_T(\cdot)$ (see Definition \ref{def: reversible sol} and Theorem \ref{thm:reversible wellposedness}).
Finally,  by Proposition \ref{prop:reversible unique extension}, we conclude that
$$
 \partial_w \mathcal{J}_T(u_0)  = 2\displaystyle\int_{\R^N} w(0,x) d\pi^0(x),
$$
where $\pi^0$ is the unique measure that extends the reversible solution $\pi$ by continuity in $\mathcal{M}(\R^N)-w^\ast$, at $t=0$.
\end{proof}

\subsection{Duality solutions}
\label{subsec: duality solutions}

In this subsection, we briefly recall the definition and main properties of duality solutions to linear transport equations with a coefficient satisfying the one-sided-Lipschitz condition.   We refer to \cite{bouchut2005uniqueness} for a more detailed presentation and the proofs of the results presented in this subsection.

We consider the linear transport equation
\begin{equation}\label{linear transport subsec}
\left\{ \begin{array}{ll}
\partial_t v + a(t,x) \cdot \nabla v = 0 &  \text{in} \ (0,T)\times \R^N \\
\noalign{\vspace{2pt}}
v(0,x)  = v_0 (x) & \text{in} \ \R^N,
\end{array}\right.
\end{equation}
where the initial condition satisfies $v_0 \in BV_{loc} (\R^N)$, and the vector field $a \in L^\infty ((0,T)\times \R^N; \R^N)$ is the so-called transport coefficient, that can have discontinuities, but is assumed to satisfy the OSLC condition
\begin{equation}\label{OSLC}
 \langle a(t,y)-a(t,x), \, y-x \rangle  \leq \alpha(t)  |y-x|^2  \qquad \text{for a.e.} \  (t,x,y)\in (0,T)\times\R^N\times\R^N,
\end{equation}
for some function $\alpha \in L^1 (0,T)$.

Note that  \eqref{OSLC} implies only an upper bound on $\text{div}_x a$, and thus, $\text{div}_x a$ may not be absolutely continuous with respect to the Lebesgue measure, preventing us from using the renormalized approach by DiPerna-Lions in \cite{diperna1989ordinary}.
The framework that we have chosen to deal with transport equations with discontinuous coefficients $a\in L^\infty ((0,T)\times \R^N)$ satisfying OSLC condition \eqref{OSLC} is the one of duality solutions, established by Bouchut-James \cite{bouchut1998one} for the one-dimensional case in space, and by Bouchut-James-Mancini in \cite{bouchut2005uniqueness} for the multidimensional case.

The main idea in \cite{bouchut2005uniqueness} to establish well-posedness for the problem \eqref{linear transport subsec} consists in solving  the dual (or adjoint) equation to \eqref{linear transport subsec}, which is a conservative transport equation of the form
\begin{equation}\label{conservative backward}
\left\{
\begin{array}{ll}
\partial_t \pi + \text{div}_x (a(t,x) \pi) = 0 & \text{in} \ (0,T)\times \R^N, \\
\noalign{\vspace{2pt}}
\pi (T,x) = \pi^T (x) & \text{in} \ \R^N.
\end{array}\right.
\end{equation}

Following the approach by Bouchut-James-Mancini in \cite{bouchut2005uniqueness}, we define the reversible solutions to \eqref{conservative backward} by using the notion of transport flow.

\begin{definition}[Transport flow]\label{def: backward flow}
Let $T>0$, and $a(\cdot,\cdot)\in L^\infty((0,T)\times \R^N)$ be given.
We say that the Lipschitz map
$$
X^T :  [0,T] \times \R^N \longrightarrow \R^N  
$$
is a \emph{backward flow} in $[0,T]\times \R$ associated to $a(\cdot, \cdot)$ if 
$$
\left\{ \begin{array}{ll}
\partial_t X^T(t,x)  + a(t,x) \cdot \nabla_x X^T(t,x) = 0  & \text{a.e.  in} \ (0,T)\times \R^N \\
\noalign{\vspace{2pt}}
X^T (T,x) = x & \text{in} \ \R^N.
\end{array}\right.
$$
and satisfies $\text{det} (\nabla_x X^T ) \geq 0$ for all $t\in [0, T]$, where 
$\nabla_x X^T$ is the distributional Jacobian of the vector field $X^T$.
\end{definition}

With the notion of transport flow, we can now define the notion of reversible solution for the conservative transport equation \eqref{conservative backward}.

\begin{definition}[Reversible solution]
\label{def: reversible sol}
We say that $\pi \in C([0,T], L_{loc}^\infty (\R^N)-w^\ast)$ is a reversible solution to \eqref{conservative backward}, if for some transport flow $X^T$ one has
$$
\pi (t,x) = \pi (T, X^T(t,x)) \text{det} (\nabla_x X^T (t,x))
$$
for all $0\leq t\leq T$ and a.e. $x\in \R^N$.
\end{definition}

In \cite{bouchut2005uniqueness}, it is proved that for any transport coefficient $a$ satisfying \eqref{OSLC}, there exists at least a transport flow, however, uniqueness cannot be ensured in the multi-dimensional case.
Nonetheless,  the property that yields uniqueness for the problem \eqref{conservative backward} is the fact that any transport flow associated to $a$ have the same Jacobian determinant.
Indeed, it can be proved that $ \text{det} (\nabla_x X^T (t,x))$ actually vanishes in the regions where $X^T$ is not uniquely determined.
Let us state the existence and uniqueness result for the backward conservative problem \eqref{conservative backward}, whose proof can be found in \cite{bouchut2005uniqueness}.

\begin{theorem}[Theorem 3.10 from \cite{bouchut2005uniqueness}]
\label{thm:reversible wellposedness}
Let $T>0$, and $a(\cdot,\cdot)\in L^\infty((0,T)\times \R^N)$ satisfying \eqref{OSLC} be given. 
For any $\pi^T \in L_{loc}^\infty (\R^N)$, there exists a unique reversible solution $\pi$ to \eqref{conservative backward} such that $\pi (T, \cdot) = \pi^T$.
Moreover, the solution can be given  by
$$
\pi (t,x) = \pi^T (X^T (t,x))  \text{det} (\nabla_x X^T (t,x)).
$$
\end{theorem}

Let us now go back to the nonconservative  transport problem \eqref{linear transport subsec}.
Due to the low regularity of the transport coefficient $a (t,x)$,
we can only expect to have solutions of bounded variation in $x$.
Let us define the space
$$
\mathcal{S}_{BV} = C([0,T], L_{loc}^1 (\R^N)) \cap \mathcal{B} ([0,T], BV_{loc} (\R^N)),
$$
where $\mathcal{B}$ stands for the space of bounded functions.
Let us now give the definition of duality solution for the transport equation \eqref{linear transport subsec}.

\begin{definition}
\label{def:duality sol}
We say that $v \in \mathcal{S}_{BV}$ is a duality solution to \eqref{linear transport subsec} if for any $0 <\tau \leq T$ and for any reversible solution $\pi \in C([0,\tau], L_{loc}^\infty (\R^N)-w^\ast)$ to \eqref{conservative backward} with compact support in $x$, it holds that
$$
t \longmapsto \displaystyle\int_{\R^N} v(t,x) \pi (t,x) dx \qquad \text{is constant in} \ [0, \tau].
$$
\end{definition}

A relevant feature of duality solutions is the following property, which corresponds to Lemma 4.2 in \cite{bouchut2005uniqueness}.

\begin{lemma}[Lemma 4.2 in \cite{bouchut2005uniqueness}]
\label{lem: duality solutions}
Let $p \in \Lip_{loc} ([0,T]\times \R^N)$ solve 
$$
\partial_t p + a \nabla_x p = 0 \qquad \text{a.e.  in } \ (0,T)\times \R^N.
$$
Then $p$ is a duality solution.
\end{lemma}

We end this subsection with the statements of the results from \cite{bouchut2005uniqueness} concerning the main properties of duality solutions to the problem \eqref{linear transport subsec}, namely, existence uniqueness and stability.

\begin{theorem}[Theorem 4.3 from \cite{bouchut2005uniqueness}]
\label{thm:duality wellposedness}
Let $T>0$, and let $a (\cdot, \cdot) \in L^\infty ((0,T)\times \R^N)$ satisfy \eqref{OSLC}.
For any $v_0\in BV_{loc} (\R^N)$, there exists a unique duality solution $v\in \mathcal{S}_{BV}$ to \eqref{linear transport subsec} such that $v(0,\cdot) = v_0$.
\end{theorem}

In order to state the  stability result, let us consider a sequence of coefficients $a_n$ such that
\begin{equation}\label{a_n}
a_n \ \text{is uniformly bounded in} \ L^\infty ((0,T)\times \R^N),
\end{equation}
and
\begin{equation}\label{a_n OSCL}
a_n \ \text{satisfies the OSLC condition \eqref{OSLC} for some $\alpha_n$ bounded in $L^1(0,T)$}.
\end{equation}
Note that \eqref{a_n} and \eqref{a_n OSCL} imply that, after the extraction of a subsequence, we have that there exists $a \in L^\infty ((0,T)\times \R^N)$ such that
\begin{equation}\label{a_n limit}
a_n \rightharpoonup a \quad \text{in} \ L^\infty((0,T)\times \R^N) - w^\ast.
\end{equation}
Moreover, in view of Lemma 2.1 in \cite{bouchut2005uniqueness}, we deduce that the limit coefficient  $a$ also satisfies the OSLC condition \eqref{OSLC}.

\begin{theorem}[Theorem 5.2 from \cite{bouchut2005uniqueness}]
\label{thm:weak stability duality}
Let $T>0$, assume \eqref{a_n}--\eqref{a_n limit}, and let $v^0_n$ be a bounded sequence in $BV_{loc}(\R^N)$ such that $v_n^0\to v_0$ in $L_{loc}^1(R^N)$.
Then the duality solution $v_n$ to
$$
\partial_t v_n + a_n \cdot \nabla_x v_n = 0 \quad \text{in} \ (0,T)\times \R^N, \qquad v_n(0,\cdot) =  v_n^0
$$
converges in $C([0,T], L^1_{loc}(\R^N))$ to a duality solution to
$$
\partial_t v + a \cdot \nabla_x v = 0 \quad \text{in} \ (0,T)\times \R^N, \qquad v(0,\cdot) = v^0.
$$
\end{theorem}

\subsection{Semiconcavity estimate}
\label{subsec:semiconcavity estimates}

In this subsection we recall a fundamental property of the viscosity solutions to Hamilton-Jacobi equations of the form \eqref{HJ eq}, which implies that the solution is semiconcave with linear modulus and constant $C/t$, for some $C>0$. This property implies in particular that the transport coefficient $a$ defined in \eqref{transport-coef thm} satisfies the OSLC condition \eqref{OSLC OSLC}, which is a key feature in the proof of Theorem \ref{thm:gateaux deriv transport}.
Let us recall that, under the hypotheses \eqref{general Hcond},\eqref{H bounded}, \eqref{H lipschitz in x} and \eqref{H lipschitz in p} on $H$,  for any initial condition $u_0\in \Lip(\R^N)$, there exists a unique viscosity solution to \eqref{HJ eq} satisfying $u(0,\cdot) = u_0$, and moreover, this solution actually coincides with the value function of an optimal control problem as follows:
\begin{equation}\label{value function general H OSLC}
u(t,x) = 
\inf_{\substack{\xi\in W^{1,1}(0,t;\R^N) \\
\xi(t) = x}}
\left\{
\int_0^t H^\ast (\xi(s),\, \dot{\xi}(s)) ds + u_0 (\xi(0)) \right\},
\end{equation}
where $H^\ast: \R^N\times \R^N\to \R$ is defined as the Legendre-Fenchel transform of $H(x,\cdot)$, i.e.
$$
H^\ast (x,q) = \max_{p\in \R^N} \left\{ p\cdot q - H(x,p)  \right\} \qquad \forall x\in \R^N.
$$
It is also well-known (see for instance Theorem A.2.6 and Corollary A.2.7 in \cite{cannarsa2004semiconcave}) that the hypotheses \eqref{general Hcond}, \eqref{H bounded} and \eqref{H lipschitz in x} on $H$ imply the following properties on $H^\ast$:
\begin{eqnarray}
&& H^\ast \in C^2(\R^N\times\R^N), \quad 0 <  H_{qq}^\ast (x,q) \leq \dfrac{1}{c_0} \  \forall x,q\in \R^N, \quad
\lim_{|q|\to \infty} \dfrac{H^\ast (x,q)}{|q|} = +\infty, \  \forall x\in \R^N \label{general H star properties} \\
&& H^\ast (x,q) \geq -C \quad \forall (x,q)\in \R^{2N} \quad \text{and} \quad H^\ast (x,0) \leq C \ \forall x\in \R^N \label{H star bounded} \\
&& | H^\ast (x,q) - H^\ast (y,q) | \leq C_{\Lip}' |x-y| \qquad \forall x,y,q\in \R^N, \quad \text{for some} \ C_{\Lip}' >0. \label{H star Lipschitz}
\end{eqnarray}

Analogously to the formula \eqref{value function general H OSLC}, for the forward viscosity solution, for any given terminal condition $u_T\in \Lip(\R^N)$, the unique backward viscosity solution $w\in \Lip([0,T]\times\R^N)$ to \eqref{HJ eq} satisfying $w(T,\cdot) = u_T$ can be given as the value function of a maximization problem as follows:
\begin{equation}\label{value function backward general H OSLC}
w(t,x) = 
\sup_{\substack{\xi\in W^{1,1}(t,T;\R^N) \\
\xi(t) = x}}
\left\{
-\int_t^T H^\ast (\xi(s),\, \dot{\xi}(s)) ds + u_T (\xi(T)) \right\}.
\end{equation}
See \cite{barron1999regularity} for more details on backward viscosity solutions.

Using the representation formula \eqref{value function general H OSLC} and the properties of $H^\ast$ in \eqref{general H star properties}, \eqref{H star bounded}, \eqref{H star Lipschitz}, it is possible to prove that the forward viscosity solution defined in \eqref{value function general H OSLC} is a semiconcave function, and the backward viscosity solution defined in \eqref{value function backward general H OSLC} is semiconvex. 

Let us recall here the definition of semiconcavity and semiconvexity with linear modulus.

\begin{definition}\label{def:semiconcavity}
\begin{enumerate}
\item A continuous function $f:\R^N\to \R$ is semicontinuous with linear modulus if there exists a constant $C>0$ such that
$$
f(x+h) + f(x-h) -2u(x) \leq C|h|^2 \quad \forall x,h\in \R^N.
$$
When this property holds,we say that $C$ is the semiconcavity constant.
\item We say that $f$ is semiconvex with linear modulus and constant $C>0$ if the function $x\mapsto -f(x)$ is semiconcave with linear modulus and constant $C$. 
\end{enumerate}
\end{definition}

\begin{remark}\label{rmk: semiconcavity}
It is easy to see that a function $f:\R^N\to \R$ is semiconcave (resp. semiconvex) with linear modulus and constant $C>0$ if and only if the function
$$
\tilde{f}(x) : = f(x) - \dfrac{C}{2}|x|^2 \quad \left( \text{resp.} \  \tilde{f}(x) := f(x) + \dfrac{C}{2}|x|^2 \right)
$$
is concave (resp. convex).
This implies that $\tilde{f}$ is locally Lipschitz and satisfies
$$
\left\langle \nabla \tilde{f}(y) - \nabla \tilde{f}(x), \, y-x \right\rangle \leq 0 \qquad \text{for a.e.} \  x,y \in \R^N,
$$
which then yields the one-sided-Lipschitz estimate
$$
\left\langle \nabla f(y) - \nabla f(x), \, y-x\right\rangle \leq C | y-x |^2, \qquad \text{for a.e.} \  x,y\in \R^N.
$$
\end{remark}

Although it is a well-known property, we give here a short proof of the semiconcavity and semiconvexity estimates for the forward and backward viscosity solutions \eqref{value function general H OSLC} and \eqref{value function backward general H OSLC}.
For further results regarding the regularity of the viscosity solutions to Hamilton-Jacobi equations, we refer the reader to \cite{bianchini2012sbv, cannarsa2014pointwise,  cannarsa1997regularity, lions2020new}.

\begin{proposition}\label{prop: semiconcavity property}
Let $T>0$ and let $H:\R^N\times\R^N\to \R$ be a given  Hamiltonian satisfying \eqref{general Hcond}, \eqref{H bounded}, \eqref{H lipschitz in x} and \eqref{H lipschitz in p}. Then, we have the following
\begin{enumerate}
\item For any $t\in (0,T]$ and $u_0\in \Lip(\R^N)$, the value function $u(t,\cdot)$ defined in \eqref{value function general H OSLC} is semiconcave with linear modulus and constant 
$\dfrac{C}{t}$, for some constant $C>0$ depending only on $H^\ast$.
\item For any $t\in [0,T)$ and $u_T\in \Lip(\R^N)$, the value function $w(t,\cdot)$ defined in \eqref{value function backward general H OSLC} is semiconvex with linear modulus and constant 
$\dfrac{C}{T-t}$, for some constant $C>0$ depending on $H^\ast$.
\end{enumerate}
\end{proposition}

\begin{proof}
We only give the proof of the semiconcavity estimate for \eqref{value function general H OSLC}, since the semiconvexity  estimate for \eqref{value function backward general H OSLC} can be proved analogously.

Let $x\in \R^N$ and $t>0$ be fixed.  By the properties of $H^\ast$ in \eqref{general H star properties}, \eqref{H star bounded} and \eqref{H star Lipschitz},  we can use the direct method of calculus of variations to prove the existence of an arc $\xi_{t,x} \in C([0,t];\R^N)$,  satisfying $\xi_{t,x}(t) = x$,  such that
\begin{equation}\label{equality semiconcavity proof}
u(t,x) = \int_0^t H^\ast (\xi_{t,x}(s), \dot{\xi}_{t,x} (s)) ds + u_0 (\xi_{t,x}(0)).
\end{equation}
Now, for any $y\in \R^N$, let us set the arc $\tilde{\xi}_{t,x} \in C([0,t];\R^N)$ defined by
$$
\tilde{\xi}_{t,x} (s) = \xi_{t,x}(s) + \dfrac{y-x}{t}s \qquad \text{for} \ s\in [0,t].
$$
Note that $\tilde{\xi}_{t,x}$ satisfies
$$
\tilde{\xi}_{t,x} (t) = y, \qquad \tilde{\xi}_{t,x}(0) = \xi_{t,x}(0) \quad \text{and}\quad \dot{\tilde{\xi}}_{t,x} (s) = \dot{\xi}_{t,x} (s) + \dfrac{y-x}{t}.
$$
Then, by  \eqref{value function general H OSLC}, we have
\begin{equation}\label{ineq semiconcavity proof}
u(t,y) \leq \int_0^t H^\ast \left( \xi_{t,x} (s) + \dfrac{y-x}{t}s,  \dot{\xi}_{t,x} (s) + \dfrac{y-x}{t} \right) ds + u_0 (\xi_{t,x}(0)).
\end{equation}

In view of the properties  \eqref{general H star properties}--\eqref{H star Lipschitz} on $H^\ast$,  for any $\varepsilon>0$ small,  there exists a constant $K>0$ depending on $\xi_{t,x}$ and  $\varepsilon$ such that
\begin{eqnarray*}
H^\ast \left( \xi_{t,x} (s) + \dfrac{y-x}{t}s,  \dot{\xi}_{t,x} (s) + \dfrac{y-x}{t} \right) & \leq &  H^\ast (\xi_{t,x}(s), \dot{\xi}_{t,x} (s)) + \dfrac{C_1}{t} 
\langle H^\ast_x (\xi_{t,x}(s), \dot{\xi}_{t,x} (s)) , \, y-x \rangle \\
& & + \dfrac{C_2}{t} \langle H^\ast_p (\xi_{t,x}(s), \dot{\xi}_{t,x} (s)) ,\,  y-x \rangle
+ \dfrac{C_3}{t^2} | x-y|^2  
\end{eqnarray*}
for all $y\in B(0,\varepsilon)$ and $s\in [0,t]$,  where $C_1,C_2,C_3>0$ are three constants depending on $H$.
Hence, combining this estimate with \eqref{equality semiconcavity proof} and \eqref{ineq semiconcavity proof}, we obtain
\begin{eqnarray*}
u(t,y) &\leq & \displaystyle\int_0^t H^\ast (\xi_{t,x}(s), \dot{\xi}_{t,x}(s)) ds + u_0 (\xi_{t,x} (0)) + \Lambda \cdot (y-x) + \dfrac{C_3}{t} |x-y|^2 \\
&=& u(t,x)  + \Lambda \cdot (y-x) + \dfrac{C_3}{t} |x-y|^2,
\end{eqnarray*}
for some vector $\Lambda\in \R^N$.
This implies that $u(t,\cdot)$ satisfies the inequality
$$
D^2 u(t,x) \leq \dfrac{C_3}{t}
$$
in the viscosity sense, which in turn implies the semiconcavity of $u(t,\cdot)$ with linear modulus and constant $C_3/t$.
\end{proof}

Combining Proposition \ref{prop: semiconcavity property} with the
hypotheses made on the Hamiltonian, we can deduce that, in some cases, the semiconcavity of the viscosity solution induces the OSLC \eqref{OSLC OSLC} on the transport coefficient $a:[0,T] \times \R^N \rightarrow \R^N$, as defined in \eqref{transport-coef thm}, i.e.
$$
a (t,x) = H_p (x, \nabla_x [S_t^+ u_0 (x)]),
$$ 
with a constant $C$ depending only on $H$ and $u_0$.

\begin{corollary}
\label{cor: OSLC}
Let $T>0$, $u_0, w \in \Lip(\R^N)$, and let, either $N=1$ and $H$ be a Hamiltonian satisfying \eqref{general Hcond}, \eqref{H bounded}, \eqref{H lipschitz in x} and \eqref{H lipschitz in p}, or let $N>1$ and $H$ be of the form \eqref{quadratic Hamiltonian section}.
Then the function $a(t,x)$ defined in \eqref{transport-coef thm} satisfies the OSLC \eqref{OSLC OSLC} with a constant $C>0$ depending only on $H$ and $u_0$.
\end{corollary}

\begin{proof}
In the quadratic case in any space dimenstion,  where $H$ has the form \eqref{quadratic Hamiltonian section}, the result follows directly from Proposition \ref{prop: semiconcavity property}, since
$$
H_p (x,p) = 2p, \qquad \forall p\in \R^N.
$$
In the one-dimensional case in space,  \eqref{OSLC OSLC} can be written as
$$
a(t,y) - a(t,x) \leq \dfrac{C}{t} (y-x) \qquad \forall x,y\in \R, \quad \text{with } \ y \geq x.
$$
We can write
\begin{eqnarray}
a(t,y)-a(t,x) & = & H_p(y, \partial_x [S_t^+u_0] (y)) - H_p( x, \partial_x [S_t^+u_0](x)) \nonumber \\
& = & H_p(y, \partial_x [S_t^+u_0] (y)) -  H_p( x, \partial_x [S_t^+u_0](y)) \\
& & +  H_p( x, \partial_x [S_t^+u_0](y)) -  H_p( x, \partial_x [S_t^+u_0](x)). \label{proof corollary OSLC conclusion}
\end{eqnarray}
Now, on one hand, we can use the Lipschitz hypotheses \eqref{H lipschitz in x} and  \eqref{H lipschitz in p} on the Hamiltonian to deduce
\begin{equation}
\label{proof corollary OSLC part 1}
H_p( x, \partial_x [S_t^+u_0](y)) -  H_p( x, \partial_x [S_t^+u_0](x)) \leq C_H (y-x) \qquad \forall x,y \in \R , \quad \text{with} \ y\geq x,
\end{equation}
for some $C_H$ depending only on $H$.
On the other hand, using Proposition \ref{prop: semiconcavity property}, we deduce that$S_t^+ u_0$ satisfies
$$
\partial_x [S_t^+ u_0] (y) \leq \partial_x [S_t^+ u_0](x) + \dfrac{C}{T} (y-x) \qquad \forall x,y\in \R, \quad \text{with} \ y\geq x,
$$
and exploiting the fact that, by the convexity of $H$, $H_p$ is a monotonically increasing function we obtain
\begin{eqnarray}
 H_p( x, \partial_x [S_t^+u_0](y)) -  H_p( x, \partial_x [S_t^+u_0](x)) & \leq &  H_p \left( x, \partial_x [S_t^+ u_0](x) + \dfrac{C}{T} (y-x) \right) \nonumber \\
 & & -  H_p( x, \partial_x [S_t^+u_0](x)) \nonumber \\
 &\leq & \dfrac{C_{[H,u_0]}\, C}{T} (y-x) \qquad  \forall x,y\in \R, \quad \text{with} \ y\geq x.
 \label{proof corollary OSLC part 2}
\end{eqnarray}
Here we used the $C^2$ regularity of the Hamiltonian, which implies that $H_p$ is locally Lipschitz, then we can choose $C_{[H,u_0]}$ depending on the Lipschitz constant of $S_t^+ u_0$, which depends only on $H$ and $u_0$.
The conclusion follows by combining \eqref{proof corollary OSLC part 1}, \eqref{proof corollary OSLC part 2} and \eqref{proof corollary OSLC conclusion}.
\end{proof}

\subsection{Proof of Theorem \ref{thm:gateaux deriv transport}}
\label{subsec: proof linearization}
In this subsection we give the proof of Theorem \ref{thm:gateaux deriv transport}, which relies on the well-posedness of the transport equation \eqref{transport equation intro}, when the transport coefficient is given by \eqref{transport-coef thm}. 
Since the coefficient $a(t,x)$ given in \eqref{transport-coef thm} satisfies the OSLC condition \eqref{OSLC} with $\alpha(t) = C/t$, which obviously does not belong to $L^1(0,T)$, we cannot directly apply the results in subsection \ref{subsec: duality solutions}. The first step in the proof of Theorem \ref{thm:gateaux deriv transport} is to prove that the limit as $\delta \to 0^+$ of the function
\begin{equation}\label{v delta}
v_\delta (t,x) := \dfrac{S_t^+ (u_0 + \delta w)(x) - S_t^+ u_0 (x)}{\delta} \qquad \text{for} \ \delta>0,
\end{equation} 
 is a duality solution (not necessarily unique) to the forward transport equation
\begin{equation}
\label{transpor-eq epsilon}
\left\{\begin{array}{ll}
\partial_t v + a(t,x)  \nabla_x v = 0 & (t,x) \in (0,T)\times \R^N \\
v(0, x) = w (x) & x\in \R^N.
\end{array}\right.
\end{equation}
Then, we will prove that equation \eqref{transpor-eq epsilon} only admits a unique duality solution. This will be proven as a consequence of Proposition \ref{prop:reversible unique extension}, which ensures that the reversible solutions to the backward conservative equation \eqref{eq: conservative backward opt cond} can be uniquely extended at $t=0$ by a measure.

\begin{proposition}\label{prop: deriv transport eq}
Let $T>0$, $u_0, w \in \Lip(\R^N)$, and let, either $N=1$ and $H$ be a Hamiltonian satisfying \eqref{general Hcond}, \eqref{H bounded}, \eqref{H lipschitz in x} and \eqref{H lipschitz in p}, or let $N>1$ and $H$ be of the form \eqref{quadratic Hamiltonian section} Then  we have that
$$
v_\delta (t,x) \to v(t,x)  \qquad \text{in} \ C([0,T], L^1_{loc} (\R^N))
$$
where $v_\delta(t,x)$ is defined for all $\delta>0$ as in \eqref{v delta}, and  $v(t,x)$ satisfies the transport equation
\eqref{transpor-eq epsilon} with transport coefficient $a(t,x)$ given by \eqref{transport-coef thm} in the duality sense of Definition \ref{def:duality sol}.
\end{proposition}

Note that, in the definition of duality solution we use, as test functions, reversible solutions $\pi \in C([0,\tau], L_{loc}^\infty (\R^N)-w^\ast)$ to the dual equation \eqref{conservative backward}. 
The fact that the transport coefficient $a(t,x)$ does not satisfy the OSLC condition \eqref{OSLC} with $\alpha\in L^1(0,T)$ implies that for some terminal conditions $\pi^\tau\in L^\infty_{loc}(\R^N)$, a reversible solution $\pi \in C([0,\tau], L_{loc}^\infty (\R^N)-w^\ast)$ satisfying $\pi(\tau, \cdot) = \pi^\tau(\cdot)$ may not exist. However, it does not represent any inconvenient in the definition of duality solution.
Existence of reversible solutions for any terminal condition are necessary to ensure the uniqueness of the duality solution, and this will be done in Proposition \ref{prop:reversible unique extension} by considering measure-valued solutions to \eqref{conservative backward}.

\begin{proof}
For any $\delta>0$ and $(t,x)\in (0,T]\times \R^N$, let us set
$$
u(t,x) = S_t^+ u_0 (x) \qquad \text{and} \qquad
 u_\delta (t,x) = S_t^+(u_0 + \delta w) ( x).
$$
We can then write
$$
v_\delta (t,x) = \dfrac{u_\delta (t,x) - u(t,x)}{\delta}
$$
The fact that, for all $t\in [0,T]$,
\begin{equation*}
v_\delta (t,\cdot) \to v(t,\cdot)  \qquad \text{in} \ L^1_{loc} (\R^N)
\end{equation*}
to some $v(t,\cdot)\in L^\infty_{loc} (\R^N)$ follows from Theorem \ref{thm: differentiability general weak convergence}. 
Moreover, from the representation formula for the limit, given in Theorem \ref{thm: differentiability general weak convergence}, and the regularity of the Hamiltonian, it follows that the map
\begin{equation}
\label{proof prop convergence}
t\in [0,T] \longmapsto v(t,\cdot)\in L^1_{loc} (\R^N)
\end{equation}
is continuous, and therefore,  we have $v\in C([0,T]; L^1_{loc} (\R^N))$.
We now need to prove that the limit $v$ is a duality solution to \eqref{transpor-eq epsilon}.

Since both $u_\delta$ and $u$ are Lipschitz continuous and verify \eqref{HJ eq} almost everywhere, we have that
\begin{equation}\label{v t deriv}
\partial_t  v_\delta (t,x) = -\dfrac{H(x,\nabla_x u_\delta (t,x)) - H(x,\nabla_x u(t,x)))}{\delta} \qquad \text{for a.e. $(t,x)\in (0,T)\times \R^N$.}
\end{equation}

Now, since $H\in C^2(\R^N\times \R^N)$, we can write
$$
H(x, \nabla_x u_\delta) = H (x,\nabla_x u) + (\nabla_x u_\delta - \nabla_x u) \cdot H_p(x,\nabla_x u) + o(|\nabla_x u_\delta - \nabla_x u|).
$$
and combining this with \eqref{v t deriv}, we deduce
$$
\partial_t  v_\delta = -\left( H_p(x,\nabla_x u) + \dfrac{o(|\nabla_x u_\delta - \nabla_x u|)}{|\nabla_x u_\delta - \nabla_x u|^2}( \nabla_x u_\delta - \nabla_x u) \right) \dfrac{\nabla_x u_\delta - \nabla_x u}{\delta}.
$$

We now set the transport coefficient $a_\delta$ as
\begin{equation}\label{a_delta}
a_\delta(t,x) := H_p(x, \nabla_x u) + \dfrac{o(|\nabla_x u_\delta - \nabla_x u|)}{|\nabla_x u_\delta - \nabla_x u|^2}( \nabla_x u_\delta - \nabla_x u).
\end{equation}
Then, since $v_\delta$ is a Lipschitz function and satisfies 
$$
\partial_t v_\delta + a_\delta(t,x)  \nabla_x v_\delta = 0  \qquad \text{for a.e.} \  (t,x) \in (0,T)\times \R,
$$
we deduce from Lemma \ref{lem: duality solutions} that $v_\delta$ is a duality solution for all $\delta\in (0,1)$, with initial condition $v_\delta (0,\cdot) = w(\cdot)$.

Now, let us note that for any $\delta \in (0,1)$,  both functions $u$ and $u_\delta$ are Lipschitz in  $[0, T]\times \R^N$, with a Lipschitz constant depending on $u_0$ and $w$, but independent of $\delta$.
Hence, in view of  \eqref{a_delta}, we have that $a_\delta$ is uniformly bounded in $L^\infty ((0,T)\times\R^N)$.
Moreover, since $\nabla_x u_\delta$ converges to $\nabla_x u$ as $\delta\to 0$  for  a.e.  $(t,x)\in (0,T)\times \R$, we deduce that  $$
a_\delta(t,x) \longrightarrow H_p(x,\nabla_x u) \qquad \text{for a.e. } \  (t,x) \in (0,T)\times\R,
$$ 
implying that 
$$
a_\delta \rightharpoonup a \quad \ \text{as} \ \delta\to 0^+ \quad \text{in} \ L^\infty((0,T)\times\R)-w^\ast.
$$
Therefore, using the stability of the duality solutions from Theorem \ref{thm:weak stability duality}, and the fact that $a$ satisfies OSCL uniformly in $(\varepsilon,T]$, for all $\varepsilon>0$,  we deduce that,
 for all $\varepsilon>0$,
\begin{equation*}
v_\delta \longrightarrow v \quad  \text{as $\delta\to 0$ in} \  C([\varepsilon,T],L_{loc}^1(\R)),
\end{equation*}
where $v$ is a duality solution to \eqref{transpor-eq epsilon} in $(\varepsilon,T]\times \R^N$. 
Finally, by letting $\varepsilon\to 0^+$, and using the continuity from \eqref{proof prop convergence}, we conclude that $v$ is a duality solution to \eqref{transpor-eq epsilon} in $[0,T]\times\R^N$.
\end{proof}

We now need to prove that the limit function from Proposition \ref{prop: deriv transport eq} is actually the unique duality solution to \eqref{transpor-eq epsilon}, which will be deduced as a consequence of the fact that the reversible solutions to the dual problem can be uniquely extended at $t=0$ by a Radon measure.
This will be done in Proposition \ref{prop:reversible unique extension} below,  and to this effect, we need the following lemma, which shows that the reversible solutions to the conservative equation \eqref{eq: conservative backward opt cond} can be represented explicitly by means of the backward characteristics associated to the transport flow generated by the transport coefficient $a(t,x)$ defined in \eqref{transport-coef thm}, or in other words, by the solutions to the backward system of ODEs \eqref{optimality system thm}.

\begin{lemma}\label{lem:reversible sol expl}
Let $T>0$, $u_0, w \in \Lip(\R^N)$, and let, either $N=1$ and $H$ be a Hamiltonian satisfying \eqref{general Hcond}, \eqref{H bounded}, \eqref{H lipschitz in x} and \eqref{H lipschitz in p}, or let $N>1$ and $H$ be of the form \eqref{quadratic Hamiltonian section}.
For any $\tau \in (0,T]$ and $\pi^\tau(\cdot) \in L_{loc}^\infty(\R^N)$, there exists a unique reversible solution $\pi \in C((0,\tau]; \ , L_{loc}^\infty (\R^N)-w^\ast)$ to the backward conservative equation
\begin{equation}\label{conservative eq lemma}
\left\{
\begin{array}{ll}
\partial_t \pi + \operatorname{div}_x  (a(t,x) \pi) = 0 & \text{in}\ (0,\tau)\times \R^N, \\
\noalign{\vspace{2pt}}
\pi(\tau, x ) = \pi^\tau & \text{in} \ \R^N,
\end{array}\right.
\end{equation}
where $a(t,x)$ is defined a.e. in $(0,\tau)\times \R^N$ as
$$
a(t,x) = H_p (x,\nabla_x [S_t^+ u_0(x)]).
$$
In addition, for all $\phi \in C_c^0 (\R^N)$ and $s\in (0,\tau]$, the reversible solution $\pi$ satisfies
\begin{equation}
\label{reversible solution lem}
\displaystyle\int_{\R^N} \phi(x) \pi(s,x) dx = \int_{\R^N} \phi\left( \Gamma_{\tau, s} (x) \right) \pi^\tau (x) dx,
\end{equation}
where $\Gamma_{\tau, s}\in L^\infty_{loc}(\R^N; \R^N)$ is defined a.e. in $\R^N$ as
$$
  \Gamma_{\tau, s} (x) := \xi_{\tau,x}(s)  \quad \text{for all $x\in \R^N$ where $S_\tau^+ u_0(\cdot)$ is differentiable,}
$$
where $(\xi_{\tau,x} (\cdot), p_{\tau,x}(\cdot))\in C^1([0,\tau];\R^N\times \R^N)$ is the unique solution to
\begin{equation}\label{optim.system lem}
\left\{\begin{array}{ll}
\dot{\xi}_{\tau,x} (s)  = H_p (\xi_{\tau,x} (s), p_{\tau,x} (s)) & s\in (0,\tau) \\
\noalign{\vspace{2pt}}
\dot{p}_{\tau,x} (s) = -H_x  (\xi_{\tau,x} (s), p_{\tau,x} (s)) & s\in (0,\tau) \\
\noalign{\vspace{2pt}}
\xi_{\tau,x} (\tau) = x \quad p_{\tau,x} (\tau) = \nabla S_\tau^+ u_0 (x).
\end{array}\right.
\end{equation}
\end{lemma}

\begin{proof}
First of all, since, by means of Corollary \ref{cor: OSLC},  the transport parameter $a(t,x)$ satisfies the OSLC condition \eqref{OSLC} with $\alpha(t) = C/t$, we can use Theorem \ref{thm:reversible wellposedness} to ensure  existence and uniqueness of a reversible solution in $(\varepsilon,\tau]\times\R^N$, and by letting $\varepsilon\to 0$, we deduce that there exists a unique reversible solution $\pi \in C((0,\tau];\,  L^\infty_{loc}(\R^N)-w^\ast)$ satisfying $\pi(\tau, \cdot)= \pi^\tau(\cdot)$. Notice that the extension of the solution at $t=0$ might not be possible in the space $L^\infty_{loc}(\R^N)$.

Let us now prove the second part of the lemma.
Let $s\in (0,\tau]$ and $\phi\in \Lip(\R^N)\cap L^\infty(\R^N)$ be fixed.
We set 
$$
u_s (\cdot) := S_s^+ u_0(\cdot),
$$
and by  the semigroup property, we have that
$$
S_{t-s}^+ u_s(x) = S_t^+ u_0 (x), \qquad\forall (t,x)\in [s,T]\times \R^N,
$$
and then, we also have
$$
a_s (t,x) = H_p (x,\,  \nabla_x [S_{t-s}^+ u_s (x)]) = a(t,x),\qquad\forall (t,x)\in [s,T]\times \R^N.
$$

Now, by means of Theorem \ref{thm:duality wellposedness}, and using the uniform OSLC in $[s,T]$, we have that, for any $\phi \in C_c^0 (\R^N)$, there exists a unique duality solution $v\in \mathcal{S}_{BV}$ to  the linear transport equation
$$
\left\{
\begin{array}{ll}
\partial_t v + a_s(t,x) \nabla_x v = 0 & \text{in} \ (s,T)\times \R^N \\
\noalign{\vspace{2pt}}
v(s, x) = \phi(x) & \text{in} \ \R^N.
\end{array}
\right.
$$
Moreover, by Proposition \ref{prop: deriv transport eq}, we have 
$$
v (t,\cdot) = \partial_\phi S_{t-s}^+ u_s(\cdot) \qquad \forall t\in [s, T],
$$
and by Theorem \ref{thm: differentiability general weak convergence}, we have in addition that
$$
v(t,x) = \phi (\xi_{t,x} (s)) \quad \forall (t,x)\in (s,T]\times \R^N \ \text{where $S_{t-s}^+u_s(\cdot)$ is differentiable at $x$,}
$$
where $\xi_{t,x}(\cdot)$ is defined\footnote{Obviously, in \eqref{optim.system lem}, we have to replace $\tau$ by $t$.  Recall moreover that $S_{t-s}^+u_s(\cdot) = S_t^+ u_0(\cdot)$.} as in \eqref{optim.system lem}, in the statement of the Lemma.

Finally, by the definition of duality solution, we conclude that
$$
\displaystyle\int_{\R^N} \phi (x) \pi (s,x) dx = \displaystyle\int_{\R^N} v (\tau, x) \pi (\tau,x) dx = \displaystyle\int_{\R^N}   \phi (\Gamma_{\tau,s}(x)) \pi^\tau (x) dx,
$$
where $\Gamma_{\tau,s}: \R^N\longrightarrow \R^N$ is defined a.e. in $\R^N$ as 
$$
\Gamma_{\tau,s}(x) =\xi_{\tau,x} (s), \qquad \forall x\in \R^N \quad
\text{such that} \ S_\tau^+u_0 (\cdot)\ \text{is differentiable.}
$$
\end{proof}

We can now use Lemma \ref{lem:reversible sol expl}
to prove that any reversible solution to \eqref{conservative backward} with compact support can be uniquely extended at $t=0$ by a finite Radon measure in $\R^N$.

\begin{proposition}\label{prop:reversible unique extension}
Under the same assumptions as in Lemma \ref{lem:reversible sol expl}, for any $\tau\in (0,T]$ and $\pi^\tau (\cdot) \in L^\infty (\R^N)$ with compact support, 
there exists a unique Radon measure $\mu \in \mathcal{M}(\R^N)$ such that the unique reversible solution $\pi \in C((0,\tau], L_{loc}^\infty (\R^N)-w^\ast)$ to \eqref{conservative eq lemma} satisfies
$$
\pi(t,\cdot) \rightharpoonup \pi^0 \qquad \text{as} \ s\to 0^+, \quad \text{in the weak}^\ast \ \text{topology of} \ \mathcal{M}(\R^N).
$$
Hence, for any $\tau\in (0,T]$ and $\pi^\tau (\cdot)\in L^\infty (\R^N)$ with compact support,  the backward conservative problem \eqref{conservative eq lemma} admits a unique measure-valued solution  $\pi^\ast  \in C([0,\tau],\mathcal{M} (\R^N)-w^\ast)$, given by
$$
\pi^\ast (s) \left\{
\begin{array}{ll}
\pi (s,\cdot) & \text{for} \ s\in (0, \tau] \\
\noalign{\vspace{2pt}}
\pi^0 & \text{for} \ s=0.
\end{array}
\right.
$$
\end{proposition}

\begin{proof}
Let us note that, by property \eqref{reversible solution lem} from Lemma \ref{lem:reversible sol expl}, along with the fact that $\pi^\tau$ is in $L^\infty (\R^N)$ and compactly supported, and that the solutions to the optimality system \eqref{optim.system lem} remain in a compact set depending only on $T$ and $\| u_0 \|_{L^\infty(\R^N)}$ (see \cite[Lemma 1]{ancona2016compactness} and also \cite{ancona2016quantitative}), one can deduce that there exists a constant $C$ such that
$$
\displaystyle\int_{\R^N} |\pi(s,  x) | dx \leq C ,\qquad \forall s\in (0,\tau].
$$
Hence, we can apply De La Vallée Poussin compactness criterion for Radon measures (see \cite[Theorem 1.59]{ambrosio2000functions}), to deduce that for any sequence $\{s_n\}_{n\geq 1}$ with $s_n\to 0^+$, the sequence of measures associated to the functions $\pi(s_n,\cdot)$ has a subsequence that converges in the $\text{weak}^\ast$ topology of $\mathcal{M}(\R^N)$.

Let us now prove that for any sequence $\{ s_n\}_{n\geq 1}$, the $\text{weak}^\ast$ limit is unique.
Let $\{ s_n\}_{n\geq 1}$ and $\{ t_n\}_{n\geq 1}$ be two sequences such that $s_n\to 0^+$ and $t_n\to 0^+$.
Then,  for any $\phi\in C_c^0(\R^N)$,by virtue of \eqref{reversible solution lem} in Lemma \ref{lem:reversible sol expl}, we have that
\begin{equation}
\label{measure Cauchy sequence}
\left| \displaystyle\int_{\R^N} \phi(x) \left[ \pi (s_n, x) dx -  \pi (t_n, x)\right] dx \right|  = 
\left| \displaystyle\int_{R^N} \left[ \phi(\Gamma_{\tau, s_n} (x)) - \phi (\Gamma_{\tau,t_n}(x) ) \right] \pi^\tau(x) dx \right|.
\end{equation}
Now,  in view of the definition of $\Gamma_{\tau,s}(\cdot)$ in the statement of Lemma \ref{lem:reversible sol expl} and by the continuity of $\phi$, we deduce that
$$
\phi(\Gamma_{\tau, s_n} (x)) - \phi (\Gamma_{\tau,t_n}(x) )  = 
\phi(\xi_{\tau, x} (s_n)) - \phi (\xi_{\tau,x}(t_n) ) \to 0 \qquad \text{as $n\to\infty$, for a.e.} \  x\in \R^N,
$$
and this, together with \eqref{measure Cauchy sequence} and the fact that $\pi^\tau\in L^\infty(\R^N)$ is compactly supported,  allows us to apply Lebesgue's dominated convergence Theorem to obtain
$$
\left| \displaystyle\int_{\R^N} \phi(x) \left[ \pi (s_n, x) dx -  \pi (t_n, x)\right] dx \right|  \to 0 \qquad \text{as} \ n\to \infty.
$$
This implies that for any sequence $\{s_n\}_{n\geq 1}$, the sequence of functions $\pi (s_n,\cdot)$ converges in the $\text{weak}^\ast$ topology to a unique Radon measure $\pi^0\in \mathcal{M}(\R^N)$.
\end{proof}

Let us conclude the section with the proof of Theorem \ref{thm:gateaux deriv transport}, which is nothing but a combination of Propositions \ref{prop: deriv transport eq} and \ref{prop:reversible unique extension}.

\begin{proof}[Proof of Theorem \ref{thm:gateaux deriv transport}]
By Proposition \ref{prop: deriv transport eq},  we have that
$$
\dfrac{S_t^+ (u_0 + \delta w) (x) - S_t^+ u_0 (x)}{\delta} \longrightarrow v(t,x) \qquad \text{in} \quad C([0,T], L_{loc}^1(\R^N)),
$$
where $v(t,x)$ is a duality solution to the linear transport equation \eqref{transpor-eq epsilon} with transport coefficient \eqref{transport-coef thm} and initial condition $v(0,\cdot) = w(\cdot)$.
We only need to prove that this is in fact the unique duality solution to \eqref{transpor-eq epsilon}.

Let $v_1, v_2 \in C ([0,T], L_{loc}^1 (\R^N))$ be two duality solutions to \eqref{transpor-eq epsilon} satisfying
$$
v_1(0,x) = v_2 (0,x) = w (x) \qquad \forall x\in \R^N.
$$

Now, for any $\tau\in (0, T]$ and any $\pi^\tau \in L^\infty (\R^N)$ with compact support, let $\pi\in C((0,\tau], \, L^\infty (\R^N)-w^\ast)$ be the unique reversible solution to \eqref{conservative backward} in $(0,\tau] \times \R^N$ with terminal condition $\pi(\tau, \cdot) = \pi^\tau(\cdot)$,
and let $\pi^0 \in \mathcal{M}(\R^N)$ be the unique Radon measure, obtained by means of Proposition \ref{prop:reversible unique extension} as the limit
$$
\pi (s, \cdot) \rightharpoonup \pi^0 \qquad \text{as} \ s\to 0^+ \quad \text{in the weak}^\ast \ \text{topology of $\mathcal{M}(\R^N)$}.
$$
Then, by the definition of duality solution we have that the maps 
$$
s \longmapsto \displaystyle\int_{\R^N} v_1(s,x) \pi(s,x)dx \qquad
\text{and} \qquad
s \longmapsto \displaystyle\int_{\R^N} v_2(s,x) \pi(s,x)dx
$$
are constant in $(0,\tau]$, and in particular,  we have
$$
 \displaystyle\int_{\R^N} v_1(\tau,x) \pi^\tau(x)dx =
  \displaystyle\int_{\R^N} v_2(\tau,x) \pi^\tau(x)dx  = 
   \displaystyle\int_{\R^N} w(x) d\pi^0(x)
$$
for all $\pi^\tau \in L^\infty (\R^N)$.
This implies that  $v_1(\tau,x) =  v_2(\tau,x)$ for a.e.  $x\in \R^N$ and for all $\tau\in (0,T]$. Note that the right-hand-side in the above equality is well-defined as $w$ is a continuous function.
\end{proof}

\section{Existence of minimizers}
\label{sec:existence minimizers}

In this section, we prove that the optimization problem 
\eqref{L2 OCP S^+}
has at least one solution.
In the proof, we shall make use of the \emph{backward viscosity operator} $S_T^-:\Lip (\R^N)\longrightarrow \Lip(\R^N)$, whose definition we recall here.
\begin{equation}\label{value function backward general H}
S_T^- u_T (x) = 
\sup_{\substack{\xi\in W^{1,1}(t,T;\R) \\
\xi(0) = x}}
\left\{
-\int_0^T H^\ast (\xi(s),\, \dot{\xi}(s)) ds + u_T (\xi(T)) \right\}, \qquad \forall x\in \R^N.
\end{equation}
Note that this is the analogous version to the forward viscosity operator $S_T^+$ defined in \eqref{value function general H},  i.e. $S_T^- u_T (\cdot)$ is the unique \emph{backward viscosity solution} at time $0$ to the Hamilton-Jacobi equation \eqref{HJ eq} with terminal condition $u(T,\cdot) = u_T(\cdot)$. See \cite{barron1999regularity} for further details on backward and forward viscosity solutions.

Let us state and prove the existence result for the optimization problem \eqref{L2 OCP S^+}.

\begin{theorem}\label{thm:existence L2 projection}
Let $T>0$ and let $H$ be a Hamiltonian satisfying \eqref{general Hcond} , \eqref{H lipschitz in x}, \eqref{H lipschitz in p}, and  \eqref{H bounded} with $C_0=0$ .
Let $S_T^+$ be the forward viscosity operator defined in \eqref{value function general H}, and let $u_T\in \Lip_0(\R^N)$ be a given function with compact support.
Then there exists a function $u_0^\ast\in \Lip_0(\R^N)$ solution to the optimization problem \eqref{L2 OCP S^+}.
\end{theorem}

\begin{proof}
It is well-known, see for instance \cite{barron1999regularity,misztela2020initial}, that combining the operators $S_T^+$ and $S_T^-$, we have following property:
\begin{equation}\label{ping pong}
S_T^+ (S_T^-( S_T^+ u_0)) = S_T^+ u_0 \qquad \forall u_0 \in \Lip(\R^N).
\end{equation}

Let $\{\varphi_n\}_{n\in\N} \subset \Lip_0(\R^N)$ be a minimizing sequence for the functional $\mathcal{J}_T$.
In view of the definition of $S_T^+$, and since $u_T$ is Lipschitz continuous with compact support, we can assume, without loss of generality, that the sequence $\varphi_n$ is equibounded and that all the elements are supported in a compact set independent of $n\in \N$.

Besides, by the property \eqref{ping pong}, the sequence of initial conditions
$$
\tilde{\varphi}_n := S_T^- (S_T^+ \varphi_n) \qquad \forall n\in \N 
$$
is also a minimizing sequence for $\mathcal{J}_T$.
Moreover, using a comparison argument and the finite speed of propagation of the equation \eqref{HJ eq}, we can deduce that the sequence $\tilde{\varphi}_n$ is also equibounded and has support in a compact set independent of $n\in \N$.

Now, we can use the regularizing effect of the backward viscosity operator $S_T^-$,  see Proposition \ref{prop: semiconcavity property}, to ensure that all the elements of the sequence $\tilde{\varphi}_n$ are semiconvex with linear modulus and constant $\frac{C}{T}$, independent of $n\in \N$. Since $\tilde{\varphi}_n$ is also equibounded and with compact support, we can deduce, using Theorem 2.1.7 and Remark 2.1.8 in \cite{cannarsa2004semiconcave}, that the sequence $\tilde{\varphi}_n$ is equicontinuous in a compact set, and then, by means of Arzéla-Ascoli Theorem, we can extract a subsequence, that we still denote by $\tilde{\varphi}_n$, that converges uniformly to some $u_0^\ast \in \Lip(\R^N)$.

Finally, using that $\tilde{\varphi}_n$ is a minimizing sequence, together with the continuity of the operator $S_T^+$ from Theorem \ref{thm: differentiability general weak convergence} (see also Remark \ref{rmk: S_T^+ Lipschitz}), we conclude that 
$$
\mathcal{J}_T (u_0^\ast) = \lim_{n\to \infty} \mathcal{J}_T (\tilde{\varphi}_n) = \inf_{u_0\in \Lip(\R^N)} \mathcal{J}_T (u_0).
$$
\end{proof}

Note that Theorem \ref{thm:existence L2 projection} provides existence of an optimal inverse design for any $u_T\in \Lip_0(\R^N)$, as the solution to the optimization problem \eqref{L2 OCP S^+}. Due to the lack of backward uniqueness for the initial-value problem \eqref{HJ eq}, uniqueness of an optimal inverse design is not in general true.
We can however consider a different but related problem, which is that of the  $L^2$-projection of $u_T$ onto the reachable set $\mathcal{R}_T$, that we can formulate as the optimization problem
\begin{equation}\label{L^2 proj existence}
\underset{\varphi\in \mathcal{R}_T}{\text{minimize}} \ \mathcal{H} (\varphi):= \|\varphi - u_T\|_{L^2(\R^N).}
\end{equation}

For the case of $x-$independent quadratic Hamiltonians of the form
\begin{equation}\label{quadratic hamiltonian}
H(p) = \dfrac{\langle A\, p,\, p\rangle}{2}, \quad \text{for some positive definite matrix $A\in \mathcal{M}_N (\R)$},
\end{equation}
we can actually prove that the optimization problem \eqref{L^2 proj existence} admits a  unique solution by means of Hilbert Projection Theorem, using a sharp characterization of $\mathcal{R}_T$ based on a semiconcavity inequality.
It is proved in \cite[Theorem 2.2]{esteve2020inverse} that a target $u_T$ is reachable in time $T>0$ if and only if $u_T$ is a viscosity solution to the second-order differential inequality
 \begin{equation}\label{semiconcavity ineq}
D^2 u_T - \dfrac{A^{-1}}{T} \leq 0 \qquad \text{in} \ \R^N,
 \end{equation}
 where $D^2 u_T$ denotes the Hessian matrix of $u_T$. 
This inequality represents, in fact, the necessary and sufficient semiconcavity condition for the reachability of a target. 
Such a precise semiconcavity condition for the characterization of the reachable set $\mathcal{R}_T$ is, up to the best of our knowledge, unavailable for non-quadratic $x$-dependent Hamiltonians.

\begin{remark}\label{rmk: convexity reachable set}
Observe that, we can use the reachability condition \eqref{semiconcavity ineq} to prove that the reachable set is convex whenever the Hamiltonian is of the form \eqref{quadratic hamiltonian}.

Indeed, note that \eqref{semiconcavity ineq} is equivalent to say that  the function
$$
x\longmapsto u_T (x) - \dfrac{\langle A^{-1} \, x, \, x\rangle}{2T}  \qquad \text{is concave.}
$$
Then, for any two functions $u_T,v_T\in \mathcal{R}_T$  and any scalar $\alpha\in (0,1)$, observe that the function
\begin{eqnarray*}
x\mapsto \alpha u_T(x) + (1-\alpha)v_T(x) -  \dfrac{\langle A^{-1} \, x, \, x\rangle}{2T}  &=& 
\alpha \left(u_T(x) - \dfrac{\langle A^{-1} \, x, \, x\rangle}{2T} \right) \\
&& + (1-\alpha)\left( v_T(x) -  \dfrac{\langle A^{-1} \, x, \, x\rangle}{2T} \right)
\end{eqnarray*}
is a concave function as it is the convex combination of two concave functions. Hence,  $\alpha u_T(x) + (1-\alpha)v_T(x)\in \mathcal{R}_T$ and we can conclude that $\mathcal{R}_T$ is convex.

\end{remark}

Let us state and prove the following result, which ensures the existence and uniqueness of solution for the optimization problem \eqref{L^2 proj existence}.

\begin{theorem}\label{thm:existence and uniqueness}
Let $T>0$,  let $H$ be a Hamiltonian of the form \eqref{quadratic hamiltonian}, and let $u_T\in \Lip(\R^N)$ be a given compactly supported function. Then, the optimization problem \eqref{L^2 proj existence} has a unique solution $\varphi^\ast\in \Lip(\R^N)\cap L^2(\R^N)$.
\end{theorem}

\begin{proof}
The existence and uniqueness of solution to problem \eqref{L^2 proj existence} follows from Hilbert Projection Theorem, after proving that $\mathcal{R}_T\cap L^2(\R^N)$ is a convex closed set in $L^2(\R^N)$.  
The convexity of $\mathcal{R}_T$ follows from Remark \ref{rmk: convexity reachable set}, which directly implies the convexity of $\mathcal{R}_T\cap L^2(\R^N)$.  Let us now verify that $\mathcal{R_T} \cap L^2(\R^N)$ is closed in $L^2(\R^N)$. Let $\{\varphi_n\}_{n=1}^\infty$ be a sequence of functions in $\mathcal{R_T} \cap L^2(\R^N)$ strongly converging to some $\varphi^\ast \in L^2(\R^N)$.
This implies that $\varphi_n(x)$ converges to $\varphi^\ast(x)$ for almost every $x\in \R^N$.
Hence,  we have that
$$
x\longmapsto \varphi_n(x) - \dfrac{\langle A^{-1} x, \, x\rangle}{2T}
$$
is a sequence of concave functions that converges for a.e. $x\in \R^N$ to the function
$$
x\longmapsto \varphi^\ast(x) - \dfrac{\langle A^{-1} x, \, x\rangle}{2T}
$$
which is therefore also a concave function, implying that $\varphi^\ast\in \mathcal{R}_T\cap L^2(\R^N)$.  We then conclude that $\mathcal{R}_T\cap L^2(\R^N)$ is a convex closed set of $L^2(\R^N)$, and the conclusion of the theorem follows. 
\end{proof}

Now, we can combine Theorem \ref{thm:existence and uniqueness} with \cite[Theorem 2.6]{esteve2020inverse} to  describe the set of all the solutions to the optimal control problem \eqref{L2  OCP S^+} when the Hamiltonian is of the form \eqref{quadratic hamiltonian}. 

\begin{corollary}\label{cor:set of inverse designs}
Let $T>0$, let $H$ be of the form \eqref{quadratic hamiltonian},  and let $u_T\in \Lip_0(\R^N)$ be a given compactly supported function.
Set the function
$$
\tilde{u}_0 (x) = S_T^- \varphi^\ast (x) := \max_{y\in \R^N} \left\{
\varphi^\ast(y) - \dfrac{\langle A^{-1}(y-x), \, y-x\rangle}{2T}
\right\},
$$
where $\varphi^\ast$ is the unique solution to \eqref{L^2 proj existence} provided in Theorem \ref{thm:existence and uniqueness}.
Then, the set of solutions to the optimal control problem \eqref{L2  OCP S^+} is the convex cone in $\Lip(\R^N)$ defined as 
$$
\mathcal{I}_T(\varphi^\ast) = \left\{ \tilde{u}_0 + \phi \, ; \ \phi \in \Lip(\R^N) \ 
\text{such that} \ \phi \geq 0 \ \text{and} \ \text{supp} (\phi) \subset \R^N \setminus X_T( \varphi^\ast)   \right\},
$$
where $X_T( \varphi^\ast)$ is the subset of $\R^N$ defined as
$$
X_T( \varphi^\ast) := \left\{ z - T A \nabla \varphi^\ast(z) \, ; \ \forall z\in \R^N \ \text{such that} \ \varphi^\ast \ \text{is differentiable at} \ z \right\}.
$$
\end{corollary}

\begin{remark}\label{rmk:structure of solutions}
Observe that this corollary establishes in particular existence of solutions for the optimal control problem \eqref{L2  OCP S^+}. However, in view of the form of $\mathcal{I}_T (\varphi^\ast)$,  the solution is unique if and only if $\varphi^\ast$ is differentiable in $\R^N$.
\end{remark}

\section{Conclusion and perspectives}

In this work, we studied the differentiability of the nonlinear operator
$S_t^+$, defined in \eqref{forward viscosity operator intro},
which associates to any initial condition $u_0\in \Lip(\R^N)$, the viscosity solution to \eqref{HJ eq}
at time $t\in [0,T]$.
First we proved that for any $t\in (0,T]$, the operator $S_t^+$ is differentiable with respect to the $L_{loc}^1$-convergence at any initial condition $u_0\in \Lip(\R^N)$ and in any direction $w\in \Lip(\R^N)$, i.e.
we prove that for any $u_0,w\in \Lip(\R^N)$, it holds that
$$
\dfrac{S_t^+ (u_0 + \delta w) (\cdot) - S_t^+ u_0 (\cdot) }{\delta} \longrightarrow \partial_w S_t^+ u_0 (\cdot) \quad \text{as} \quad \delta\to 0^+, \quad \text{in} \quad L^1_{loc} (\R^N),
$$
where the function $\partial_w S_t^+ u_0 (\cdot)\in L^\infty_{loc}(\R^N)$ can be explicitly computed at all the points where $S_t^+ u_0(\cdot)$ is differentiable.  Hence, since $S_t^+ u_0(\cdot)$ is Lipschitz, and thus, differentiable for almost every $x\in \R^N$, the characterization provided in Theorem \ref{thm: differentiability general weak convergence} allows to explicitly determine the Gâteux derivative  $\partial_w S_t^+ u_0(\cdot)$ as a function in $L^\infty_{loc}(\R^N)$.

Then, for the one-dimensional case in space, and for quadratic Hamiltonians of the form 
$$H(x,p) = |p|^2 + f(x),$$
 we proved that, for any $u_0,w\in \Lip(\R^N)$, the function
$$
v(t,x) := \partial_w S_t^+ u_0(x),  \qquad \text{for} \ (t,x)\in (0,T]\times \R^N,
$$
is the unique duality solution to the linear transport equation with discontinuous transport coefficient given by
\begin{equation}
\label{transport coeff conclusions}
a(t,x) := H_p (x, \nabla_x S_t^+ u_0 (x)),
\end{equation}
and initial condition $v(0,\cdot) = w$.
The proof of this result relies on the theory of duality solutions for transport equations with discontinuous coefficient, developed by Bouchut-James in \cite{bouchut1998one} and by Bouchut-James-Mancini in \cite{bouchut2005uniqueness} for the multi-dimensional case.
The key ingredient to prove existence and uniqueness of a solution by duality is the fact that the transport coefficient $a(t,x)$ satisfies the one-sided Lipschitz condition \eqref{OSLC intro}, which can be ensured in the one-dimensional case and for quadratic Hamiltonians in any space dimension . However, the fact that the function $\alpha(t)$ in \eqref{OSLC intro} is not integrable in $(0,T)$ prevents us from directly using the results in \cite{bouchut2005uniqueness}.
In order to ensure the uniqueness of the duality solution, we prove that the backward solutions to the conservative dual equation can be extended (by continuity with respect to the weak star topology in the space of measures) at $t=0$ by a unique Radon measure. 
This unique extension provides uniqueness for the duality solution to the non-conservative forward equation, under the requirement that the initial condition is continuous.

Then we address the inverse design problem
\begin{equation}\label{OCP conclusions}
\underset{u_0\in \Lip_0(\R^N)}{\text{minimize}} \ \mathcal{J}_T (u_0) := \| S_T^+ u_0 - u_T\|_{L^2(\R^N)}^2,
\end{equation}
for some given target $u_T\in \Lip_0(\R^N)$.
The differentiability results obtained in this paper allow us to compute the gradient of the functional $\mathcal{J}_T(\cdot)$ by duality as
$$
D \mathcal{J}_T (u_0): \ w (\cdot) \longmapsto \partial_w \mathcal{J}_T(u_0) = 2 \displaystyle\int_{\R^N}  w(x) d\pi^0 (x),
$$
where $\pi^0$ is the unique Radon measure that extends at $t=0$ the backward solution to the conservative dual equation with terminal condition $S_T^+ u_0 - u_T$.

The computation of the gradient of $\mathcal{J}_T$ allows us to derive a necessary first-order optimality condition for the problem \eqref{OCP conclusions}, as well as the implementation of gradient-based methods in order to numerically approximate an optimal inverse design.
However, the fact that the gradient of $\mathcal{J}_T$ is a Radon measure prevents us from updating the initial condition in the exact opposite direction to the gradient, since it may exit the space of Lipschitz functions. Nonetheless, one can implement a modification of the gradient descent algorithm in which, at each step, the initial condition is updated in the opposite direction to a suitable Lipschitz approximation of the gradient of $\mathcal{J}_T$.
Finally, we include a section where we discuss the existence and uniqueness of optimal inverse designs for the optimization problem \eqref{OCP conclusions}.

\textbf{Open questions.} Here we give a list of question that we did not address in the present paper, and are left for future work.

\begin{enumerate}
\item[1.] The first open question is the possibility of extending the conclusion of Theorem \ref{thm:gateaux deriv transport} to the case of general convex Hamiltonians in any space dimension.
In this case, we cannot exploit the fact that the transport coefficient $a(t,x)$ defined in \eqref{transport coeff conclusions} satisfies the one-sided-Lipschitz condition, and hence, uniqueness of a duality solution to the linearized Hamilton-Jacobi equation cannot be ensured by means of our arguments.  Nonetheless, an alternative proof of the well-posedness of the transport equation might be possible by using the fact that the viscosity solution can be written as the value function of a problem of calculus of variations.

\item[2.] The results of this paper can be used in the context of optimal control problems subject to Hamilton-Jacobi equations of the form \eqref{HJ eq}, where the control is just the initial condition.
However, one may also consider optimal control problems subject to a Hamilton-Jacobi equation of the form
\begin{equation}\label{HJ eq open problems}
\left\{\begin{array}{ll}
\partial_t u + H (x,\nabla_x u, g) = 0 & \text{in} \ (0,T)\times \R^N \\
u(0, \cdot)= u_0 & \text{in} \ \R^N.
\end{array}\right.
\end{equation}
where $g\in\mathcal{G}_{ad}$ is a control parameter to be optimized, and $\mathcal{G}_{ad}$ is the given space of admissible controls.
In this context, one should study the sensitivity of the viscosity solution with respect to variations of $g$, which affect the Hamiltonian.

\item[3.] Our differentiability results apply to the case when the Hamiltonian is smooth and uniformly convex. 
These hypotheses are needed as they provide semiconcavity estimates for the viscosity solution, and these are crucial to establish the well-posedness of the transport equation resulting from the linearization of the Hamilton-Jacobi equation (recall that we need the transport coefficient to satisfy the OSLC condition).
Nonetheless, the differentiability of the viscosity solution with respect to the initial condition seems to be feasible also under less regularity assumptions.

Consider for instance the case of $x$-independent Hamiltonians $H(p)$  under the mere assumption that the map $p\mapsto H(p)$ is convex\footnote{This case includes non-smooth and non-strictly convex Hamiltonians such as $H(p) = \|p\|$, where $\|\cdot\|$ is any norm in $\R^N$.}.
In this case, the viscosity solution to the associated evolutionary Hamilton-Jacobi equation can be given by the Hopf-Lax formula \cite{alvarez1999hopf,bardi1984hopf} as
$$
u(t,x) = \min_{y\in \R^N} \left\{ u_0 (y) + t\, H^\ast \left( \dfrac{x-y}{t}  \right) \right\}.
$$
In view of this formula, the differentiability of the viscosity solution with respect to the initial condition $u_0$ might be addressed by using ideas related to Danskin's Theorem (see \cite{bernhard1995theorem, danskin1966theory}). However, the fact that the Hamiltonian is not assumed to be smooth nor strictly convex makes semiconcavity estimates unavailable, and then, it is not clear whether the theory of duality solutions from \cite{bouchut1998one,bouchut2005uniqueness} can be used to study the associated linear transport equation. 
\end{enumerate}

\bibliographystyle{abbrv}
\bibliography{mybibfile-HJ}

\end{document}